\documentclass[10pt]{article}

\usepackage{microtype}

\usepackage{amsmath,amssymb,amsthm,mathtools}
\usepackage{enumitem}
\usepackage[margin=0.82in]{geometry}
\usepackage{xcolor}
\usepackage{tikz}
\usepackage{hyperref}

\hypersetup{
    colorlinks=true,     linkcolor=blue,     citecolor=blue,     urlcolor=blue,
    pdftitle={A Hadamard Formula for Equilibrium Envelopes under Parallel Deformation},
    pdfauthor={Ziyu Li} }

\newtheorem{theorem}{Theorem}[section]
\newtheorem{proposition}[theorem]{Proposition}
\newtheorem{lemma}[theorem]{Lemma}
\newtheorem{corollary}[theorem]{Corollary}
\newtheorem{mainresult}{Theorem}

\theoremstyle{definition}
\newtheorem{definition}[theorem]{Definition}
\newtheorem{assumption}[theorem]{Assumption}
\newtheorem{remark}[theorem]{Remark}
\newtheorem*{remark*}{Remark}

\newtheorem{convention}[theorem]{Convention}
\newcommand{\ddc}{dd^c}
\newcommand{\dc}{d^c}
\newcommand{\PSH}{\mathrm{PSH}}
\newcommand{\MA}{\mathrm{MA}}
\newcommand{\tr}{\mathrm{tr}}

\newcommand{\Vol}{\mathrm{Vol}}

\newcommand{\Tr}{\mathrm{Tr}}

\DeclareMathOperator{\DM}{DM}

\newcounter{dynexample}
\renewcommand{\thedynexample}{\arabic{dynexample}}

\newenvironment{dynexample}[1]{%
	\refstepcounter{dynexample}%
	\subsubsection*{Example~\thedynexample: #1}%
	\addcontentsline{toc}{subsubsection}{Example~\thedynexample: #1}%
}{}

\setlist{leftmargin=*,itemsep=2pt,topsep=2pt}
\allowdisplaybreaks
\begin{document}

\title{A Hadamard Formula for Equilibrium Envelopes under Parallel Deformation}
\author{Ziyu Li\\Yau Mathematical Sciences Center, Tsinghua University\\\texttt{liziyu20@mails.tsinghua.edu.cn}}
\date{}
\maketitle

\begin{abstract}
Let $(X,\omega_0)$ be a compact K\"ahler manifold of complex dimension $n$, and let $U\Subset X$ have $C^{3,1}$ uniformly strongly pseudoconvex boundary. Consider the signed parallel family $U_t=\{\rho<t\}$, where $\rho$ is a defining function agreeing near $\partial U$ with the signed distance, negative in $U$. Let $u_t$ be the equilibrium envelope with zero obstacle on $X\setminus U_t$, and write $\Sigma_t=\partial U_t$, with outward unit normal $\nu_t$. We prove that the normalized Monge--Amp\`ere energy $e(t)=E_{\omega_0}(u_t)$ belongs to $C^{1,1/2}([-\tau,\tau])$ for some $\tau>0$ and satisfies the Hadamard formula
\[
e'(t)=\frac{1}{n+1}\int_{\Sigma_t}\left(-\partial_{\nu_t}u_t\right)^2\,d\sigma_{\rho,t}^{\mathrm{tot}},
\qquad |t|<\tau.
\]
Here $d\sigma_{\rho,t}^{\mathrm{tot}}$ is the interior outward Anzellotti trace of
\[
\frac{1}{V}\,\dc\rho\wedge \sum_{p=0}^{n-1}(p+1)\, \omega_{u_t}^{p}\wedge\omega_0^{n-1-p},
\qquad V=\int_X\omega_0^n,\quad \omega_{u_t}=\omega_0+\ddc u_t.
\]
This measure is positive and uniformly comparable to the induced surface measure.

The proof combines uniform $C^{1,1}$ estimates up to the boundary from the interior with $1/2$-H\"older stability of the normal derivatives under the normal-flow identifications. We identify the boundary part of the normalized Monge--Amp\`ere measure with the negative outward trace of a divergence-measure flux current. The one-sided weak-* limits of the mixed Bedford--Taylor boundary measures yield the weighted total flux appearing in the variation formula.
\end{abstract}

\medskip
\noindent\textbf{2020 Mathematics Subject Classification.} Primary 32W20; Secondary 32U15, 32Q15.

\medskip
\noindent\textbf{Key words and phrases.} complex Monge--Amp\`ere equation; equilibrium envelope; pluripotential theory; Hadamard formula; shape derivative; divergence-measure current.

\section{Introduction and the main results}\label{sec:intro}

\subsection{Moving envelopes and boundary flux}

Berman and Lu proved that, under suitable assumptions, the zero-temperature limit of the twisted K\"ahler--Ricci flow is described by a family of time-dependent Monge--Amp\`ere envelopes in their pioneering work \cite{BermanLu2018MovingFreeBoundaries}.  The non-contact regions of the corresponding envelopes, as they evolve with the parameter, give rise to moving free boundaries for an obstacle problem. In particular, on a compact Riemann surface, a singular twist arising from point divisors recovers the weak Hele--Shaw flow after a proper time reparametrization. In higher dimensions, this phenomenon corresponds to the so-called Monge--Amp\`ere growth along effective divisors. Ross and Witt Nystr\"om further established connections between Hele--Shaw flow and holomorphic discs in a series of elegant works \cite{RossWittNystrom2015HeleShawDiscs,RossWittNystrom2015HarmonicDiscs,RossWittNystrom2017PrescribedSingularities,RossWittNystrom2019HCMAHeleShaw}. Their works linked Hele--Shaw envelopes with the homogeneous complex Monge--Amp\`ere equation, prescribed singularities, and weak geodesic rays under suitable settings.

Free boundaries appear in the problems they discuss as soon as one prescribes singularities, sources, and envelopes. The evolution of these free boundaries is a central problem in pluripotential theory. In this paper, we want to approach the same issue from a different perspective: given a proper interface, how do the associated quantities change as it moves in parallel? The quantity of interest to us in this context is Monge--Amp\`ere energy. One reason for our choice comes from the variational result of Berman--Boucksom on equilibrium energy: they fixed the compact set $K$, varied the weight function $\phi$ and provided an expression for the energy derivative in terms of the equilibrium Monge--Amp\`ere measure. In our setting the weight is fixed and the constraint set varies, which leads us to study the measures associated with the moving interface and their normal traces.

The work of Berman--Lu \cite{BermanLu2018MovingFreeBoundaries} also mentions Darcy's law, which in fact provides intuition for our problem. Let $t$ denote the time parameter, $D_t$ the fluid domain, $\Gamma_t=\partial D_t$ its boundary, and $\nu_t$ the corresponding outward unit normal. If $u_t$ is the pressure and $V_t$ the normal velocity, then Darcy's law \cite{SaffmanTaylor1958Penetration,Richardson1972HeleShawFreeBoundary, 	GustafssonVasiliev2006HeleShaw} tells us that
\[
V_t=-\partial_{\nu_t}u_t.
\]
The relation $V_t=-\partial_{\nu_t}u_t$, together with the classical Hadamard formula below, suggests that a term of the form $(\partial_{\nu_t}u_t)^2$ will appear in our formula. Note that in our setting, it is the envelope function that plays the role of pressure---and in the notation that follows, we will also denote it by $u$.

\subsection{Hadamard's variational principle}

Now let us consider our problem from a different perspective. We wish to study the variation of energy, which is a type of generalized volume. It is well known that the derivative of volume is surface area, and the derivative of a generalized volume is a generalized surface measure such as flux. This intuitive idea has guided the development of many results in PDEs.

Let us first look at the case of Dirichlet energy. Let $U\subset\mathbb R^N$ be a smooth bounded domain, and define
\[
U_\varepsilon=(\operatorname{id}+\varepsilon W)(U), \qquad V_n:=W\cdot\nu.
\]
Let $v_\varepsilon$ satisfy
\[
-\Delta v_\varepsilon=f \quad\text{in }U_\varepsilon, \qquad v_\varepsilon=0 \quad\text{on }\partial U_\varepsilon,
\]
where $f$ is fixed in a neighbourhood of $\overline U$. The corresponding Dirichlet energy
\[
\mathcal D(U_\varepsilon) := \int_{U_\varepsilon} \left( \frac12|\nabla v_\varepsilon|^2-fv_\varepsilon \right)\,dx
\]
satisfies
\[
\left.\frac{d}{d\varepsilon}\right|_{\varepsilon=0} \mathcal D(U_\varepsilon) = -\frac12\int_{\partial U} (\partial_\nu v_0)^2V_n\,dS.
\]
For details, see \cite{SokolowskiZolesio1992Shape,HenrotPierre2005Shape}.

The proof of the formula proceeds primarily through three mechanisms. Firstly, one can use the PDE to eliminate the variation of state in the interior so that the first variation only relies on contributions from the boundary. Secondly, differentiating the boundary condition produces the normal derivative. Finally, by the smoothness of the solution and the classical Green formula, one obtains a flux density with respect to the boundary measure. In our degenerate complex Monge--Amp\`ere setting, the first two mechanisms use pluripotential theory and differential-geometric techniques. The third uses the theory of divergence-measure fields to describe boundary fluxes \cite{Anzellotti1983Pairings,ChenFrid1999Divergence,ChenComiTorres2019CauchyFluxes}.

Hadamard originally described this phenomenon using the language of Green functions. Let $G_U(\,\cdot\,,y)$ be the Dirichlet Green function normalized by $-\Delta G_U(\,\cdot\,,y)=\delta_y$ and $G_U(\,\cdot\,,y)|_{\partial U}=0$. Hadamard showed that \cite{Hadamard1908Memoire,SuzukiTsuchiya2016HadamardGreen}
\[
\left.\frac{d}{d\varepsilon}\right|_{\varepsilon=0} G_{U_\varepsilon}(x,y) = \int_{\partial U} \partial_{\nu_\xi}G_U(x,\xi)\, \partial_{\nu_\xi}G_U(y,\xi)\, (W\cdot\nu)(\xi)\,dS(\xi).
\]
The product of the two normal derivatives in this formula motivates the mixed terms in our final formula.

Our final formula is
\[
\frac{d}{dt}E_{\omega_0}(u_t) = \frac1{n+1}\int_{\Sigma_t} \left(-\partial_{\nu_t}u_t\right)^2\,d\sigma_{\rho,t}^{\mathrm{tot}}.
\]
The left-hand side is the first variation of the energy under consideration. On the right, the term $\left(-\partial_{\nu_t}u_t\right)^2$ is the square of the velocity anticipated earlier, and $d\sigma_{\rho,t}^{\mathrm{tot}}$ is a boundary surface measure formulated in the language of divergence-measure fields. The prefactor $\frac1{n+1}$ arises from the mixed terms, but these mixed terms are eventually absorbed into the expression of $d\sigma_{\rho,t}^{\mathrm{tot}}$.

Related developments concern the structure of shape derivatives, irregular domains, capacity, and nonlinear elliptic equations. The following results illustrate these directions: Novruzi and Pierre discussed the structure of the first-order shape derivative, namely that it depends only on the normal component of the boundary deformation \cite{NovruziPierre2002Structure}. Lamboley and Pierre extended such structural results to irregular domains \cite{LamboleyPierre2007Irregular}. Colesanti, Nystr\"om, Salani, Xiao, Yang, and Zhang established a Hadamard variational formula for the $p$-capacity \cite{ColesantiNystromSalaniXiaoYangZhang2015PCapacity}. Brandolini, Nitsch, and Trombetti studied shape optimization problems for the real Monge--Amp\`ere equation using the method of domain derivatives \cite{BrandoliniNitschTrombetti2011MongeAmpereShape}.

\subsection{Prescribed-interface setting}
Let $(X,\omega_0)$ be a compact K\"ahler manifold of complex dimension $n$, and let $g_0$ be the metric defined by $\omega_0$.  We write
\[
V=\int_X\omega_0^n,\qquad \dc=\frac{i}{2}(\bar\partial-\partial),\qquad \ddc=i\partial\bar\partial,\qquad \omega_u=\omega_0+\ddc u.
\]
Next, we discuss the geometric setting. We fix a domain $U\Subset X$ whose boundary satisfies the following assumption.
\begin{assumption}[Uniformly strongly pseudoconvex $C^{3,1}$ boundary]\label{ass:BH}
$\Sigma:=\partial U$ is a compact embedded $C^{3,1}$ strongly pseudoconvex real hypersurface.
\end{assumption}
Let $d_\Sigma$ be the signed $g_0$-distance to $\Sigma$, negative in $U$.  If $\delta_0$ is small enough, then on the collar $\mathcal N_{2\delta_0}(\Sigma):=\{\operatorname{dist}_{g_0}(x,\Sigma)<2\delta_0\}$ we have
\[
d_\Sigma\in C^{3,1},\qquad |\nabla_{g_0}d_\Sigma|_{g_0}=1 \quad\text{on }\mathcal N_{2\delta_0}(\Sigma).
\]
We fix a global defining function $\rho\in W^{2,\infty}(X)$ such that
\[
\rho=d_\Sigma\ \text{on }\mathcal N_{\delta_0}(\Sigma),\qquad U=\{\rho<0\},\qquad \Sigma=\{\rho=0\},\qquad \Omega=X\setminus U=\{\rho\ge0\},
\]
and $|\rho|$ is bounded away from zero outside $\mathcal N_{\delta_0}(\Sigma)$.  Then $\dc\rho\in W^{1,\infty}(X)$ and $\ddc\rho\in L^\infty(X)$.

For $|t|<\delta_0$, all level sets share the same collar around $\Sigma$.  We set
\[
\rho_t:=\rho-t,\qquad U_t=\{\rho_t<0\}=\{\rho<t\},\qquad \Sigma_t=\{\rho=t\},\qquad \Omega_t=X\setminus U_t,\qquad \nu_t=\nabla_{g_0}\rho|_{\Sigma_t},
\]
Then $\nu_t$ is the outward unit normal to $U_t$, and we put $\nu_{\mathrm{out}}:=\nu_0$.  

Next, we turn to the function setting.  The central objects of pluripotential theory are the plurisubharmonic functions, or PSH functions for short.  A function $u\colon X\to[-\infty,+\infty)$ is called $\omega_0$-plurisubharmonic if it is upper semicontinuous, not identically equal to $-\infty$, and, for every local potential $\psi$ of $\omega_0$, the function $u+\psi$ is plurisubharmonic, which can be denoted $u \in \PSH(X,\omega_0)$. More details and the properties of Bedford--Taylor products are given in Section~\ref{sec:preliminaries}.

We define
\[
u_t:=\sup\{v\in\PSH(X,\omega_0):v\le0         \text{ on }\Omega_t\}^{*}, \qquad q_t:=-\partial_{\nu_t}u_t.
\]
Here ${}^{*}$ denotes upper semicontinuous regularization.  By the standard Perron--Bremermann construction \cite{Bremermann1959GeneralizedDirichlet,Demailly2012Complex,BermanDemailly2012Envelopes}, $u_t\in\PSH(X,\omega_0)$, $u_t=0$ on $\Omega_t$, and $\omega_{u_t}^n=0$ on $U_t$ in the Bedford--Taylor sense.  Since $\dc\rho_t=\dc\rho$, recentering the defining function at level $t$ preserves the reference fluxes below.

For bounded $\omega_0$-psh functions we set
\begin{equation}\label{eq:intro-MA-energy-definitions}
\MA(u):=\frac1V\omega_u^n,\qquad E_{\omega_0}(u):=\frac1{(n+1)V}\sum_{j=0}^{n} \int_Xu\,\omega_u^j\wedge\omega_0^{n-j}.
\end{equation}

For Theorems~\ref{intro:thm-mixed-traces}--\ref{intro:thm-hadamard} we write
\[
B_t^{\mathrm{tot}}:=\sum_{p=0}^{n-1}(p+1) \omega_{u_t}^p\wedge\omega_0^{n-1-p},\qquad \mathcal J[\rho_t;\mathrm{tot}]:=\frac1V\dc\rho\wedge B_t^{\mathrm{tot}}.
\]
We will prove that this current has an interior outward trace on $\Sigma_t$, which we denote by $d\sigma_{\rho,t}^{\mathrm{tot}}:=\Tr_{\Sigma_t}^{\mathrm{out}}\mathcal J[\rho_t;\mathrm{tot}]$.

Next, we make some remarks about measures.  We write $\Delta_{\omega_0}f=\tr_{\omega_0}(\ddc f)$, so that $\Delta_{g_0}f=2\Delta_{\omega_0}f$.  We write $\mu\llcorner A$ for the restriction of a measure, and $\mathcal M(\Sigma_t)$ for the finite signed Radon measures on $\Sigma_t$.  All traces in this paper are interior outward Anzellotti traces; see Section~\ref{sec:preliminaries} for the details.  Unless otherwise indicated, $C=C(X,\omega_0,U,\rho)$ is uniform in $t$, $\varepsilon$, and $\delta$, and may differ from line to line.  Table~\ref{tab:notation} collects the notation used repeatedly in this paper.

\begin{convention}[Standing parameters]\label{intro:conv-parameters}
	We fix the small parameters in the order
	\[
	\delta_0 \;>\; \varepsilon_0 \;>\; \varepsilon_1 \;>\; 2\tau \;>\; 0,
	\]
	and each of them may be taken sufficiently small.  The number $\delta_0$ is the collar radius; $\varepsilon_0$ appears in Proposition~\ref{prop:BH_to_collar} and is the range of the uniform collar; $\varepsilon_1$ is the half-width of the parallel family in Theorem~\ref{intro:thm-uniform-regularity}, and it is decreased in Section~\ref{reg:sec-regularity} (see Convention~\ref{reg:conv-uniform-collar}); $\tau$ is the half-width appearing in Section~\ref{sec:hadamard}.  The parameters $\varsigma_0,\varsigma$ appear only in Appendix~\ref{app:sec-proofs-regularity}.
\end{convention}

\begin{table}[!htbp]
\centering
\small
\caption{Principal recurring notation}
\label{tab:notation}
\setlength{\tabcolsep}{5pt}
\renewcommand{\arraystretch}{1.05}
\begin{tabular}{@{}llll@{}}
\multicolumn{2}{@{}l}{\textbf{Geometry and envelopes}} & \multicolumn{2}{l}{\textbf{Fluxes and traces}}\\
$V$ & volume $\int_X\omega_0^n$ & $\mathcal J[u]$ & flux current $V^{-1}\dc u\wedge T_u$\\
$\rho$, $d_\Sigma$ & defining function; signed distance & $\mathcal J[\rho;c]$ & weighted reference flux current\\
$\rho_t,U_t,\Sigma_t,\Omega_t$ & $\rho-t$, $\{\rho_t<0\}$, $\{\rho_t=0\}$, $X\setminus U_t$ & $\mathcal J[u_0;\mathrm{tot}]$ & total flux current (weights $c_p=p+1$)\\
$\nu_t$ & outward unit normal to $\Sigma_t$ & $B_c$, $B_t^{\mathrm{tot}}$ & mixed wedge sums\\
$u_t$ & equilibrium envelope of $\Omega_t$ & $\mathcal J_\varepsilon^{(j)},\mathcal J_0^{(j)}$ & mixed flux currents\\
$u_{t,\delta}$ & positive-density Dirichlet approximant & $\Tr_{\Sigma}^{\mathrm{out}}$ & interior outward trace\\
$q_t$ & $-\partial_{\nu_t}u_t$ & $d\sigma_\rho$, $d\sigma_{\rho,t}^{\mathrm{tot}}$ & reference flux traces\\
$\omega_u$ & $\omega_0+\ddc u$ & $\sigma_\varepsilon^{(j)},\sigma_{0+}^{(j)}$ & mixed boundary traces\\
$\MA(u)$ & $V^{-1}\omega_u^n$ & $\sigma_{\mathrm{tot}}$ & $\sum_{j=0}^{n-1}\sigma_{0+}^{(j)}$\\
$E_{\omega_0}(u)$ & Monge--Amp\`ere energy & $\mu_0^{\mathrm{sing}}$ & singular boundary mass of $\MA(u_0)$\\
$\delta_0,\varepsilon_0,\varepsilon_1,\tau$ & standing parameters & $\DM^\infty(U)$ & divergence-measure currents\\
\end{tabular}
\end{table}

\subsection{Main results}
Theorem~\ref{intro:thm-hadamard} is the central result of this paper.  To obtain it, we first need the PDE result Theorem~\ref{intro:thm-uniform-regularity}, which gives the uniform interior regularity and moving-domain stability needed to define and compare the boundary fluxes.  Theorem~\ref{intro:thm-boundary-mass} connects pluripotential theory with geometric measure theory: it identifies the singular boundary mass with a weak normal trace.  Theorem~\ref{intro:thm-mixed-traces} identifies the mixed Monge--Amp\`ere measures.

Fix $\varepsilon_1>0$ so that Proposition~\ref{prop:BH_to_collar} and Lemma~\ref{prelim:lem-parallel-identifications} hold uniformly for $|t|\le\varepsilon_1$. The first theorem gives the interior Hessian bounds and moving-level stability we need, and these provide the regularity needed to place the later fluxes in the divergence-measure class. 
\begin{mainresult}[Uniform regularity and stability]\label{intro:thm-uniform-regularity}
There is $C>0$, depending only on $(X,\omega_0,U,\rho)$, such that for $|t|\le\varepsilon_1$,
\[
\|u_t\|_{C^{1,1}(\overline{U_t})} +\|\Delta_{\omega_0}u_t\|_{L^\infty(U_t)}\le C, \qquad -\omega_0\le\ddc u_t\le C\omega_0 \quad\text{a.e. on }U_t\text{, from the interior side}.
\]
For every $0<\beta<1$ and all $s,t$ in this interval,
\[
\|u_t-u_s\|_{L^\infty(X)}\le C|t-s|, \qquad \|u_t\circ\Xi_{s,t}-u_s\|_{C^{1,\beta}(\overline{U_s})} \le C_\beta|t-s|^{(1-\beta)/2},
\]
where $\Xi_{s,t}$ is the normal-flow identification of Lemma~\ref{prelim:lem-parallel-identifications}.  If $u_{t,\delta}$ is the positive-density Dirichlet approximant (the solution of \eqref{reg:eq-approx-moving} with $\varepsilon=t$), then
\[
\|u_{t,\delta}-u_t\|_{L^\infty(U_t)}\le C\delta^{1/n}, \qquad \|u_{t,\delta}-u_t\|_{C^{1,\beta}(\overline{U_t})} \le C_\beta\delta^{(1-\beta)/(2n)}.
\]
In particular, with $\Psi_{s,t}=\Xi_{s,t}|_{\Sigma_s}$,
\[
\|q_t\circ\Psi_{s,t}-q_s\|_{C^0(\Sigma_s)} \le C|t-s|^{1/2}.
\]
See Propositions~\ref{reg:lem-uniform-inner-C11} and~\ref{reg:prop-quantitative-delta-domain-stability}.
\end{mainresult}

The exponent $1/2$ appears again in Theorem~\ref{intro:thm-hadamard}.

\begin{mainresult}[Singular boundary mass as normal trace]\label{intro:thm-boundary-mass} 
Let $\mu_0=\MA(u_0)$ and $\mu_0^{\mathrm{sing}}:=\mu_0\llcorner\partial U$. Define
\[
T_{u_0}=\sum_{p=0}^{n-1}\omega_{u_0}^{p}\wedge\omega_0^{n-1-p}, \qquad \mathcal J[u_0]=\frac1V\dc u_0\wedge T_{u_0}, \qquad \mathcal J[\rho]=\frac1V\dc\rho\wedge T_{u_0},
\]
where the currents are defined on $U$ using the interior Bedford--Taylor representatives.  The geometric reference trace $d\sigma_\rho:=\Tr_{\partial U}^{\mathrm{out}}\mathcal J[\rho]$ is a positive Radon measure, and
\[
\mu_0^{\mathrm{sing}}=-\Tr_{\partial U}^{\mathrm{out}}\mathcal J[u_0]=\left(-\partial_{\nu_{\mathrm{out}}}u_0\right)d\sigma_\rho
\]
as an identity of Radon measures on $\partial U$; see Propositions~\ref{prop:jump-density} and~\ref{stat:prop-boundary-mass}.
\end{mainresult}

The quantity $\mu_0^{\mathrm{sing}}$ comes from pluripotential theory: it describes the singular mass on the boundary.  The quantity $-\Tr_{\partial U}^{\mathrm{out}}\mathcal J[u_0]$ which is a trace in the sense of divergence-measure fields comes from geometric measure theory.  And $\left(-\partial_{\nu_{\mathrm{out}}}u_0\right)d\sigma_\rho$ is a normal slope times a positive geometric hypersurface measure, giving a differential-geometric expression for the boundary mass.  Theorem~\ref{intro:thm-boundary-mass} relates these three with each other.

For $0<\varepsilon\le\varepsilon_1$ and $0\le j\le n-1$, set
\[
\mu_{0,\varepsilon}^{(j)} :=\frac1V\omega_{u_\varepsilon}^{j}\wedge\omega_{u_0}^{n-j}, \qquad \sigma_\varepsilon^{(j)} :=\mu_{0,\varepsilon}^{(j)}\llcorner\partial U.
\]
For $0\le j\le n-1$, define the mixed flux current
\[
\mathcal J_0^{(j)} :=\frac1V\dc u_0\wedge\omega_{u_0}^{j}\wedge \sum_{k=0}^{n-j-1}\omega_0^{k}\wedge\omega_{u_0}^{n-j-1-k}.
\]

Using this notation, we obtain the following result, which is parallel to Theorem~\ref{intro:thm-boundary-mass}:

\begin{mainresult}[Mixed traces and total flux]\label{intro:thm-mixed-traces}
	The currents above use the interior Bedford--Taylor representatives on $U$.
	\begin{enumerate}[label=\textup{(\roman*)}]
		\item For every $0\le j\le n-1$, the current $\mathcal J_0^{(j)}$ admits an outward trace, and
		\[
		\sigma_{0+}^{(j)}:=-\Tr_{\partial U}^{\mathrm{out}}\mathcal J_0^{(j)}
		\]
		is a Radon measure satisfying
		\[
		\sigma_\varepsilon^{(j)}\stackrel{*}{\rightharpoonup}\sigma_{0+}^{(j)} 		\quad\text{in }\mathcal M(\partial U) 		\quad\text{as }\varepsilon\downarrow0.
		\]
		
		\item Set
		\[
		\sigma_{\mathrm{tot}}:=\sum_{j=0}^{n-1}\sigma_{0+}^{(j)}, 		\qquad 		\mathcal J[u_0;\mathrm{tot}] 		:=\frac1V\dc u_0\wedge B_0^{\mathrm{tot}}, 		\qquad 		d\sigma_\rho^{\mathrm{tot}}:=d\sigma_{\rho,0}^{\mathrm{tot}}.
		\]
		Then $d\sigma_\rho^{\mathrm{tot}}$ is positive and
		\[
		\sigma_{\mathrm{tot}} 		=-\Tr_{\partial U}^{\mathrm{out}}\mathcal J[u_0;\mathrm{tot}] 		=\left(-\partial_{\nu_{\mathrm{out}}}u_0\right)d\sigma_\rho^{\mathrm{tot}}.
		\]
	\end{enumerate}
	See Theorem~\ref{thm:mixed_trace_stability} and Proposition~\ref{prop:total_jump_density}.
\end{mainresult}

The static and total reference traces $d\sigma_\rho$ and $d\sigma_\rho^{\mathrm{tot}}$, together with the corresponding boundary measures $\mu_0^{\mathrm{sing}}$ and $\sigma_{\mathrm{tot}}$, admit explicit positive densities uniformly comparable to the $g_0$-surface measure; see Proposition~\ref{stat:prop-explicit-parallel-density} and Corollary~\ref{stat:cor-parallel-density-bounds}. We remind the reader that $\sigma_{0+}^{(j)}$ is formed by taking a one-sided boundary limit as $\varepsilon\downarrow0$, whereas $\mu_0^{\mathrm{sing}}$ is the boundary restriction at $\varepsilon=0$. In higher dimensions, the sum of the one-sided limits carries the multiplicity weights $p+1$ in $d\sigma_\rho^{\mathrm{tot}}$; Example~\ref{dyn:ex-toric-collapse} computes the resulting distinction explicitly.

Using Theorem~\ref{intro:thm-mixed-traces} and the boundary expansion mechanism described above,
\[
u_\varepsilon-u_0 =-\varepsilon\,\partial_{\nu_{\mathrm{out}}}u_0+o(\varepsilon) \qquad\text{uniformly on }\partial U,
\]
we obtain:
\begin{mainresult}[Hadamard formula along the signed family]\label{intro:thm-hadamard}
	For some $\tau>0$, the map $t\mapsto E_{\omega_0}(u_t)$ belongs to 	$C^{1,1/2}([ -\tau,\tau])$, and
	\[
	\frac{d}{dt}E_{\omega_0}(u_t) 	=\frac1{n+1}\int_{\Sigma_t} 	\left(-\partial_{\nu_t}u_t\right)^2d\sigma_{\rho,t}^{\mathrm{tot}}, 	\qquad -\tau<t<\tau,
	\]
	with the corresponding one-sided formulas at the endpoints.  Equivalently,
	\[
	E_{\omega_0}(u_t)-E_{\omega_0}(u_s) 	=\int_s^t\frac1{n+1}\int_{\Sigma_r} 	\left(-\partial_{\nu_r}u_r\right)^2 	d\sigma_{\rho,r}^{\mathrm{tot}}\,dr.
	\]
	See Theorem~\ref{dyn:thm-signed-family-hadamard}.
\end{mainresult}

We stress the role of degeneracy in this variation. For a smooth path $\varepsilon\mapsto u_\varepsilon$, the first-variation formula is
\[
\frac d{d\varepsilon}E_{\omega_0}(u_\varepsilon) =\frac1V\int_X\dot u_\varepsilon\,\omega_{u_\varepsilon}^n.
\]
For the moving envelopes, exact finite differences and one-sided mixed boundary traces yield the coefficient $1/(n+1)$, which equals $1/2$ when $n=1$.

At the end of the paper, we present two examples. Example~\ref{dyn:ex-cp1-verification} checks the formula by direct computation on $\mathbb P^1$ and gives a concrete interpretation of the boundary terms. Example~\ref{dyn:ex-toric-collapse} treats a higher-dimensional toric model. Its separate energy and flux calculations verify the multiplicity weights in the total variational trace.

\paragraph{Organization.} Section~\ref{sec:preliminaries} fixes the geometric and weak-trace tools, Section~\ref{reg:sec-regularity} proves regularity and stability, and Sections~\ref{sec:flux}--\ref{sec:hadamard} derive the boundary traces and Hadamard formula.  Appendix~\ref{app:sec-proofs-regularity} contains the two long uniform-estimate proofs.

\subsection*{AI Declaration}
The author used DeepSeek-V3.2, the preview version of DeepSeek-V4-Pro, and DeepSeek-V4-Pro-0813 to assist with English-language drafting and editing, to condense an earlier draft comprising over 60 pages of handwritten and \LaTeX{} material into a manuscript of just over 30 pages, and to organize and format the bibliography. The author is responsible for the mathematical statements, proofs, references, and final text.

\section{Preliminaries}\label{sec:preliminaries}
\subsection{Parallel collar geometry}
We record the collar estimates used throughout the paper.

\begin{proposition}\label{prop:BH_to_collar}
There exists $0<\varepsilon_0<\delta_0$ such that $\rho\in C^{3,1}$ on $\{|\rho|<\varepsilon_0\}$, $|\nabla\rho|_{g_0}=1$, and the hypersurfaces $\Sigma_t=\{\rho=t\}$ are uniformly $C^{3,1}$ and uniformly strongly pseudoconvex.  Their geometric features, such as normal projection, second fundamental forms, and surface measures, are all uniformly controlled, and
\[
\Vol_{\omega_0}\{a\le\rho\le b\}\le C|b-a|.
\]
More precisely, strong pseudoconvexity is oriented with respect to $U_t$: for some $\kappa>0$ independent of $t$,
\[
\ddc\rho(\tau,\bar\tau)\ge\kappa\,\omega_0(\tau,\bar\tau), \qquad \tau\in T^{1,0}\Sigma_t.
\]
\end{proposition}
\begin{proof}
The $C^k$ version of the assertion was proved by Foote, who also pointed out that it applies to the case of Riemannian manifolds \cite[Theorem~1 and Remarks~(1)--(2)]{Foote1984Distance}.  For the $C^{3,1}$ endpoint needed here, one further bootstrap based on Gauss' lemma is required.  The key point is that a local $C^{3,1}$ parametrization has a $C^{2,1}$ unit normal.  Let $\phi(y)$ and $\nu(y)$ denote the corresponding parametrization and unit normal; then the normal exponential map $F(y,r)=\exp_{\phi(y)}(r\nu(y))$ and its inverse are $C^{2,1}$ on a smaller collar.  Gauss' lemma \cite[Chapter~3]{doCarmo1992Riemannian} gives $\nabla d_\Sigma(F(y,r))=\partial_rF(y,r)$, which yields $\nabla d_\Sigma\in C^{2,1}$ and hence $d_\Sigma\in C^{3,1}$.  The remaining part of the proof follows from compactness, the openness of strong pseudoconvexity, and the coarea formula $\Vol\{a\le\rho\le b\}=\int_a^b \mathrm{Area}(\Sigma_t)\,dt$ \cite[Theorem~3.2.22]{Federer1969GMT}.

\end{proof}

\begin{lemma}
\label{prelim:lem-parallel-identifications}
After decreasing $\varepsilon_0$, there are $C^{2,1}$ diffeomorphisms
\[
\Xi_{s,t}:\overline{U_s}\longrightarrow\overline{U_t}, \qquad \Psi_{s,t}:=\Xi_{s,t}|_{\Sigma_s}:\Sigma_s\longrightarrow\Sigma_t, \qquad |s|,|t|<\varepsilon_0,
\]
such that $\Xi_{t,t}=\operatorname{Id}$ and, in the fixed finite atlas,
\begin{equation}\label{prelim:eq-parallel-identification-bounds}
\|\Xi_{s,t}-\operatorname{Id}\|_{C^2} +\|\Xi_{s,t}^{-1}-\operatorname{Id}\|_{C^2} \le C|t-s|.
\end{equation}
The boundary map $\Psi_{s,t}$ is the normal projection along the signed-distance collar.
\end{lemma}
\begin{proof}
In signed collar coordinates $(y,r)$, translate $r=s$ to $r=t$ and splice this translation to the identity in a fixed inner subcollar.  Uniform $C^{3,1}$ collar bounds give \eqref{prelim:eq-parallel-identification-bounds}, and the boundary restriction is the normal projection.
\end{proof}
These comparison maps transport scalar functions and boundary forms. The complex Monge--Amp\`ere equations are evaluated on the original domains in their holomorphic coordinates.

\begin{remark}
Strong pseudoconvexity supplies the tangential boundary barriers.  The one-sided barriers require less regularity, but $C^{3,1}$ keeps the boundary Hessian estimates and moving-domain constants uniform along the parallel family.
\end{remark}

\subsection{Compact K\"ahler pluripotential conventions}
We use Bedford--Taylor products and their monotone continuity \cite{BedfordTaylor1976Dirichlet,BedfordTaylor1982Capacity}, together with the compact K\"ahler capacity and energy framework \cite{GuedjZeriahi2005IntrinsicCapacities,GuedjZeriahi2007WeightedEnergy}. The conventions below follow \cite{Demailly2012Complex,GuedjZeriahi2012Dirichlet,GuedjZeriahi2017Degenerate}.

\begin{definition}[$\omega_0$-psh functions]
A function $u:X\to[-\infty,+\infty)$ belongs to $\PSH(X,\omega_0)$ if $u\in L^1(X)$, $u$ is upper semicontinuous, and $u+\phi$ is plurisubharmonic on every chart where $\omega_0=\ddc\phi$.  We write $h^*$ for the upper semicontinuous regularization of $h$.
\end{definition}
\begin{definition}[Bedford--Taylor products]
For bounded $\omega_0$-psh functions $u_1,\ldots,u_m$, all mixed products
\[
\omega_{u_1}^{j_1}\wedge\cdots\wedge\omega_{u_m}^{j_m}, \qquad j_1+\cdots+j_m\le n,
\]
are defined by the local Bedford--Taylor construction after adding a local potential for $\omega_0$; the resulting currents agree on overlaps and are measures when the total degree is $n$.  We set
\[
\MA(u):=\frac{\omega_u^n}{V},\qquad \MA_j(u,v):=\frac1V\omega_u^j\wedge\omega_v^{n-j}.
\]
\end{definition}
A locally bounded $u\in\PSH(D,\omega_0)$ is \emph{locally maximal} if, for every $G\Subset D$ and every $v\in\PSH(G,\omega_0)\cap L^\infty(G)$ with
\[
\limsup_{G\ni x\to\zeta}(v-u)(x)\le0\qquad(\zeta\in\partial G),
\]
one has $v\le u$ on $G$.
\begin{theorem}
Mixed Bedford--Taylor products converge weakly as currents when each factor is replaced by a pointwise decreasing sequence of locally bounded $\omega_0$-psh functions with locally bounded limit.  If $u_k,u\in\PSH(D,\omega_0)\cap L^\infty_{\rm loc}(D)$, $u_k\uparrow u$ almost everywhere, and $(u_k)$ is locally uniformly bounded, then
\[
(\omega_0+\ddc u_k)^n\rightharpoonup(\omega_0+\ddc u)^n
\]
weakly as Radon measures.  Moreover, $u\in\PSH(D,\omega_0)\cap L^\infty_{\rm loc}(D)$ is locally maximal on $D$ if and only if $(\omega_0+\ddc u)^n=0$ on $D$.
\end{theorem}
The monotone-continuity statements are \cite[Theorem~2.4 and Proposition~5.2]{BedfordTaylor1982Capacity}; the maximality criterion is \cite[Theorem~2.20]{GuedjZeriahi2012Dirichlet}, applied after adding a local potential of $\omega_0$.
\begin{theorem}[Comparison principle]
Let $D\Subset X$ and let $u,v\in\PSH(D,\omega_0)\cap L^\infty(D)$.  If $\liminf_{D\ni x\to\zeta}(u-v)(x)\ge0$ for all $\zeta\in\partial D$, then
\[
\int_{\{u<v\}}(\omega_0+\ddc v)^n\le\int_{\{u<v\}}(\omega_0+\ddc u)^n.
\]
\end{theorem}
This is \cite[Theorem~4.1]{BedfordTaylor1982Capacity}; its local-potential proof applies unchanged on a K\"ahler domain.
\begin{corollary}\label{prelim:cor-homogeneous-dirichlet-comparison}
If $u\in\PSH(D,\omega_0)\cap L^\infty(D)$ satisfies $(\omega_0+\ddc u)^n=0$ and $w\in\PSH(D,\omega_0)\cap L^\infty(D)$ has $\limsup_{D\ni x\to\zeta}(w-u)(x)\le0$ on $\partial D$, then $w\le u$ on $D$.
\end{corollary}
Indeed, $(\omega_0+\ddc u)^n=0\le(\omega_0+\ddc w)^n$, so this is the homogeneous case of the domination principle \cite[Corollary~2.19]{GuedjZeriahi2012Dirichlet}.
\begin{lemma}[Maximum gluing]\label{prelim:lem-max-gluing}
If $G\Subset X$, $u\in\PSH(G,\omega_0)$, $w\in\PSH(X,\omega_0)$, and $\limsup_{G\ni x\to\zeta}u(x)\le w(\zeta)$ on $\partial G$, then the function equal to $\max\{u,w\}$ on $G$ and to $w$ on $X\setminus G$ belongs to $\PSH(X,\omega_0)$.
\end{lemma}
After adding a local potential of $\omega_0$, this is the usual psh pasting lemma on complex lines.
\begin{lemma}[Weak maximum principle]\label{prelim:lem-linear-weak-max}
If $D\Subset X$, $w\in L^1(D)$, and $\Delta_{\omega_0}w\ge0$ distributionally, then its canonical upper semicontinuous subharmonic representative satisfies
\[
\sup_Dw^*\le\sup_{\zeta\in\partial D}\limsup_{D\ni x\to\zeta}w^*(x).
\]
\end{lemma}
The operator $\Delta_{\omega_0}$ is a positive constant multiple of the fixed Laplace--Beltrami operator. Its weak maximum principle applied to the canonical subharmonic representative gives the boundary bound.

\subsection{Constraint sets and equilibrium envelopes}
We use the equilibrium envelope $u_t$ for the constraint set $\Omega_t$, as defined in Section~\ref{sec:intro}.
\begin{lemma}
For each $|t|<\varepsilon_0$, $u_t\in\PSH(X,\omega_0)\cap L^\infty(X)$, $u_t\ge0$, $u_t=0$ on $\Omega_t$, and $u_s\le u_t$ whenever $-\varepsilon_0<s\le t<\varepsilon_0$.
\end{lemma}
Normalized compactness and the upper-envelope theorem place $u_t$ in $\PSH(X,\omega_0)\cap L^\infty(X)$ \cite[Chapter~I, Proposition~5.9 and Theorem~5.7]{Demailly2012Complex}.  The zero competitor shows that $u_t\ge0$, and regularity of the $C^{3,1}$ compact set $\Omega_t$ gives $u_t=0$ there \cite[\S3.2, p.~360]{BermanBoucksom2010EnergyEquilibrium}.  If $s\le t$, then $\Omega_s\supset\Omega_t$, so $u_s\le u_t$.
\begin{proposition}[Envelope maximality]\label{prelim:prop-envelope-ma}
For every $|t|<\varepsilon_0$, $u_t$ is locally maximal on $U_t$.  Consequently
\[
(\omega_0+\ddc u_t)^n=0\quad\text{on }U_t.
\]
\end{proposition}
\begin{proof}
The maximum-gluing argument in the proof of \cite[Proposition~2.9]{GuedjZeriahi2012Dirichlet} applies locally: a competitor on $G\Subset U_t$ can be pasted to $u_t$ while preserving the constraint on $\Omega_t$.  Hence $u_t$ is locally maximal, and the equation follows from \cite[Theorem~2.20]{GuedjZeriahi2012Dirichlet}.
\end{proof}

\subsection{Monge--Amp\`ere energy}
For bounded $\omega_0$-psh $u$, the normalized Monge--Amp\`ere energy $E_{\omega_0}(u)$ is defined by \eqref{eq:intro-MA-energy-definitions}.
For a smooth path $u_t$ of $\omega_0$-psh functions,
\[
\frac{d}{dt}E_{\omega_0}(u_t)=\int_X\dot u_t\,\MA(u_t),
\]
so $E_{\omega_0}$ is a primitive of the Monge--Amp\`ere operator.  This is why its variation is the natural quantity in the Hadamard formula.
\begin{lemma}[Finite-difference identity]\label{prelim:lem-am-telescoping}
For bounded $\omega_0$-psh functions $u,v$,
\[
E_{\omega_0}(u)-E_{\omega_0}(v)=\frac1{n+1}\sum_{j=0}^{n}\int_X(u-v)\MA_j(u,v).
\]
\end{lemma}
This is \cite[Corollary~4.2]{BermanBoucksom2010EnergyEquilibrium}, with the normalization by $V$ used here.

\subsection{A \texorpdfstring{$\DM^\infty$}{DM-infinity} toolkit}\label{prelim:subsec-dm}
We use the Gauss--Green and normal-trace theory of \cite{Anzellotti1983Pairings,ChenFrid1999Divergence,ChenComiTorres2019CauchyFluxes}; see also \cite{ChenTorres2021Survey}. Contraction with the volume form identifies bounded $(2n-1)$-forms whose distributional exterior derivatives are finite Radon measures with divergence-measure fields, allowing us to recover their interior outward traces from inner parallel levels.

\begin{definition}[$\DM^\infty$ fields]
For $D\subset\mathbb R^N$ open,
\[
\DM^\infty(D):=\{F\in L^\infty(D;\mathbb R^N):\operatorname{div}F\in\mathcal M(D)\}.
\]
Here $\mathcal M(D)$ denotes the finite signed Radon measures and the divergence is distributional.  On a Riemannian manifold we use the equivalent localized definition.
\end{definition}
\begin{theorem}[Anzellotti trace]
Let $\Omega\Subset D$ have Lipschitz boundary.  If $F\in \DM^\infty(D)$, then there is a unique trace $[F\cdot\nu]\in L^\infty(\partial\Omega)$ such that
\[
\int_{\partial\Omega}\varphi[F\cdot\nu]dH^{N-1}=\int_{\Omega}\varphi\,d(\operatorname{div}F)+\int_{\Omega}F\cdot\nabla\varphi dx
\]
for every $\varphi\in C_c^1(D)$.  The normal $\nu$ is the outward measure-theoretic unit normal, and $[F\cdot\nu]$ is the density of the trace measure with respect to $H^{N-1}|_{\partial\Omega}$.
\end{theorem}
For smooth $F$ this is the classical normal component.  The result and its deformation form follow from \cite[Theorems~1.2 and 1.9]{Anzellotti1983Pairings}, \cite[Theorem~2.2]{ChenFrid1999Divergence}, and \cite[Theorem~8.3 and Corollary~8.1]{ChenComiTorres2019CauchyFluxes}.  All later traces are interior outward traces from the fixed domain under consideration.
\begin{theorem}[Inner parallel recovery]\label{prelim:thm-innerparallel}
Let $G\Subset D$ have $C^2$ boundary, let $r$ be its signed distance, negative in $G$, and set
\[
G_{-t}:=\{r<-t\},\qquad \nu_{-t}:=\nabla r|_{\partial G_{-t}}.
\]
For small $t>0$, the inward normal map $\Psi_t(x)=x-t\nu(x)$ identifies $\partial G$ with $\partial G_{-t}$.  If $F\in\DM^\infty(D)$, there is a full-measure set $T_F\subset(0,t_0)$ such that, for every $\varphi\in C(\partial G)$,
\[
\int_{\partial G}\varphi[F\cdot\nu]dH^{N-1} =\operatorname*{ess\,lim}_{\substack{t\downarrow0\\ t\in T_F}} \int_{\partial G_{-t}}(\varphi\circ\Psi_t^{-1})(F\cdot\nu_{-t})dH^{N-1}.
\]
In the terminology of Chen--Frid, $(\Psi_t)$ is a regular deformation: in every boundary graph chart its tangential derivatives converge in $L^1_{\rm loc}$ to those of the boundary parametrization.  For $t\in T_F$, coarea selects levels on which a precise representative of $F$ has an $H^{N-1}$-almost-everywhere restriction. The boundary object in the recovery formula is the weak normal trace defined by the Gauss--Green identity.
\end{theorem}
This is \cite[Theorems~2.1--2.2]{ChenFrid1999Divergence}; see also the Lipschitz-deformation formulation \cite[Theorem~8.4]{ChenComiTorres2019CauchyFluxes}.
\begin{lemma}[Product rule]\label{lem:product}
If $F\in \DM^\infty(D)$ and $g\in W^{1,\infty}(D)$, then $gF\in \DM^\infty(D)$ and, for every Lipschitz domain $\Omega\Subset D$,
\[
[gF\cdot\nu]=g|_{\partial\Omega}[F\cdot\nu] \qquad H^{N-1}\text{-a.e. on }\partial\Omega.
\]
\end{lemma}
This is \cite[Theorem~3.1]{ChenFrid1999Divergence}; the trace identity follows by using $g\varphi$ in the Gauss--Green formula.
\begin{lemma}[Hodge dictionary]\label{lem:hodge}
On an oriented Riemannian $N$-manifold, every $(N-1)$-form $J$ with $L^\infty$ coefficients can be uniquely written as $J=\iota_FdV_g$ with $F\in L^\infty$.  Moreover $dJ=(\operatorname{div}_gF)dV_g$.
\end{lemma}
\begin{proof}
The pointwise isomorphism $F\mapsto\iota_FdV_g=*F^\flat$ gives the unique bounded field. For $\varphi\in C_c^\infty$,
\[
\langle dJ,\varphi\rangle=-\int d\varphi\wedge J=-\int F(\varphi)\,dV_g=\langle(\operatorname{div}_gF)dV_g,\varphi\rangle.
\]
\end{proof}
\begin{lemma}[Local reduction]\label{lem:local}
In a coordinate chart, let $\chi$ be a compactly supported smooth cutoff.  The field $\widetilde F^i=\sqrt{\det g}\chi F^i$ satisfies
\[
\operatorname{div}_{\rm Eucl}\widetilde F=\chi\sqrt{\det g}\operatorname{div}_gF+\sqrt{\det g}F^i\partial_i\chi.
\]
Thus the Euclidean trace theorem localizes to the K\"ahler setting.
\end{lemma}
\begin{proof}
Write $dV_g=\gamma\,dx$, where $\gamma=\sqrt{\det g}$, and let $\mu=dJ$. The distributional identity $\partial_i(\gamma F^i)=\mu$ gives
\[
\operatorname{div}_{\rm Eucl}(\gamma\chi F)=\chi\mu+\gamma F^i\partial_i\chi\,dx.
\]
Thus $\gamma\chi F\in\DM^\infty$. Its Euclidean traces agree on chart overlaps through the common Gauss--Green functional and therefore glue by a partition of unity. The smooth identity $\iota_{\partial U}^*J=(F\cdot\nu_g)dS_g$ fixes the outward sign.
\end{proof}
\begin{theorem}[Global normal trace]\label{thm:global-trace}
Let $J$ be a real $(2n-1)$-form on $U$ with $L^\infty$ coefficients, and suppose that the top-degree distribution $dJ$ is a finite signed Radon measure.  Then $J$ admits a unique outward trace $\Tr_{\partial U}^{\mathrm{out}}J\in\mathcal M(\partial U)$ characterized as follows: for every $\varphi\in C^1(\partial U)$ and every $C^1$ extension $\widetilde\varphi$ to $\overline U$,
\[
\int_{\partial U}\varphi\,d\bigl(\Tr_{\partial U}^{\mathrm{out}}J\bigr)=\int_U\widetilde\varphi\,d(dJ)+\int_Ud\widetilde\varphi\wedge J.
\]
The right-hand side is independent of the chosen extension.  The trace is absolutely continuous with respect to the surface measure, and its density, denoted by the same symbol, satisfies
\[
\bigl\|\Tr_{\partial U}^{\mathrm{out}}J\bigr\|_{L^\infty(\partial U)} \le C_\Sigma\|J\|_{L^\infty(U)}.
\]
Here $C_\Sigma$ depends only on the fixed boundary atlas and the background metric.  The trace is recovered by inner parallel hypersurfaces in the essential-limit sense of Theorem~\ref{prelim:thm-innerparallel} and satisfies the product rule with $W^{1,\infty}$ multipliers.
\end{theorem}
\begin{proof}
Lemmas~\ref{lem:hodge}--\ref{lem:local} reduce the statement, chart by chart, to \cite[Theorem~2.2]{ChenFrid1999Divergence} and \cite[Theorem~8.4 and Corollary~8.1]{ChenComiTorres2019CauchyFluxes}.  The Gauss--Green characterization forces the local traces to agree on overlaps, so a partition of unity defines the global trace.  Summing over the fixed finite atlas yields the stated $L^\infty$ bound.  Inner recovery and the multiplier rule follow chartwise from Theorem~\ref{prelim:thm-innerparallel} and Lemma~\ref{lem:product}; uniqueness reconciles the resulting formulas on overlaps.
\end{proof}

\section{Uniform estimates on parallel domains}\label{reg:sec-regularity}
After decreasing $\varepsilon_1>0$ (Convention~\ref{intro:conv-parameters}), set
\[
I:=[-\varepsilon_1,\varepsilon_1], \qquad U_\varepsilon=\{\rho<\varepsilon\},\quad \Sigma_\varepsilon=\{\rho=\varepsilon\},\quad \Omega_\varepsilon=X\setminus U_\varepsilon, \qquad \varepsilon\in I.
\]
The positive-density Dirichlet problem is solved first on smooth approximations of the parallel domains. Uniform estimates then allow $m\to\infty$ at fixed $(\varepsilon,\delta)$, followed by $\delta\downarrow0$ at fixed $\varepsilon$.

\subsection{Smooth approximation and stability}
After the geometric approximation and low-order bounds, Lemma~\ref{reg:lem-uniform-collar-core-bridge} supplies the inner estimate used in Proposition~\ref{reg:prop-uniform-degenerate-boundary}. Appendix~\ref{app:sec-proofs-regularity} proves these two estimates. Proposition~\ref{reg:conv-c31-smoothing-reduction} then passes to the original domains.

\begin{lemma}[Smooth approximation of the parallel geometry]\label{reg:lem-smooth-parallel-geometry}
Fix $0<\alpha<1$.  Choose $\eta>\varepsilon_1$ so that $N_{3\eta}:=\{|\rho|<3\eta\}$ lies in the signed-distance collar, and put $N_{2\eta}:=\{|\rho|<2\eta\}$.  There exist $\rho_m\in C^\infty(N_{3\eta})$ with $\rho_m\to\rho$ in $C^{3,\alpha}(\overline{N_{2\eta}})$ and
\[
\sup_m\|\rho_m\|_{C^{3,1}(\overline{N_{2\eta}})}<\infty.
\]
Set
\[
U_{\varepsilon,m}:=\{\rho<-\eta\}\cup \{x\in N_{2\eta}:\rho_m(x)<\varepsilon\}, \qquad \varepsilon\in I.
\]
Then, for all large $m$ and uniformly in $\varepsilon\in I$:  $\partial U_{\varepsilon,m}=\{\rho_m=\varepsilon\}$;  $|d\rho_m|_{g_0}$ and the Levi form of $\partial U_{\varepsilon,m}$ have positive lower bounds; and the $\partial U_{\varepsilon,m}$ admit boundary charts from one fixed finite holomorphic atlas and a common tubular radius.  Gradient flow of $\rho_m$, extended by the identity off $N_{2\eta}$, yields diffeomorphisms $\Phi_{\varepsilon,m}:\overline{U_\varepsilon}\to\overline{U_{\varepsilon,m}}$ converging to the identity in $C^{3,\alpha}$, uniformly in $\varepsilon\in I$.
\end{lemma}

\begin{proof}
	We mollify $\rho$ in a fixed finite holomorphic collar atlas and patch the local mollifications together with a partition of unity; this yields a globally smooth $\rho_m\in C^\infty(N_{3\eta})$.  From the basic properties of convolution, $\rho_m$ inherits all the regularity statements listed above, and for $m$ large enough one has $\partial U_{\varepsilon,m}=\{\rho_m=\varepsilon\}\Subset N_\eta$.  To compensate for the error caused by changing the defining function, we then translate along a short normal flow.  Take $X_m=\nabla\rho_m/|d\rho_m|_{g_0}^2$; its flow satisfies $\rho_m\circ\mathrm{Fl}^t_{X_m}=\rho_m+t$, and accordingly we set $\Phi_{\varepsilon,m}(x):=\mathrm{Fl}^{\chi(x)(\rho-\rho_m)(x)}_{X_m}(x)$, where $\chi\in C^\infty_c(N_{2\eta})$ is the cutoff function equal to $1$ on $\{|\rho|\le\varepsilon_1\}$.
\end{proof}

\paragraph{Uniform low-order bounds.}
For fixed $\varepsilon\in I$, $0<\delta\le1$, and sufficiently large $m$, the smooth Dirichlet solution $u_{\varepsilon,\delta,m}$ has density $\delta$ and zero boundary value. For $n\ge2$, its existence follows from \cite[Theorem~1.1]{GuanLi2010TotallyReal}, with reference form $\omega_0$ and subsolution $0$, using the normalization specified below. For $n=1$, the equation is $\Delta_{g_0}u=2(\delta-1)$, and smooth Dirichlet solvability follows from \cite[Theorem~1.6]{TaylorNonlinearEllipticNotes}. The following bounds are uniform in $(\varepsilon,\delta,m)$.
	For $n\ge2$, we use \cite[Prop.~3.1 and (3.1)]{GuanLi2010TotallyReal}, which bounds the global gradient in terms of the zero-order bound and the boundary gradient on a compact Hermitian manifold with smooth boundary. In their convention the potential term is $(\sqrt{-1}/2)\partial\bar\partial v$, so we set $v=2u_{\varepsilon,\delta,m}$ and $\chi=\omega=\omega_0$, with $\psi=\delta$. The estimate depends on the $C^1$ norm of $\psi^{1/n}=\delta^{1/n}$, which is bounded for $0<\delta\le1$, and is uniform as $\delta\downarrow0$. To build the zero-order bound, observe that $\Delta_{\omega_0}u_{\varepsilon,\delta,m}\ge-n$, and take $H$ solving
	\[
	\Delta_{\omega_0}H=-n\quad\text{in }\widehat U, 	\qquad H=0\quad\text{on }\partial\widehat U,
	\]
	where $\widehat U\Subset X$ is a fixed smooth domain containing every $U_{\varepsilon,m}$. Comparison with the zero function gives $u_{\varepsilon,\delta,m}\ge0$, since its Monge--Amp\`ere density is $\delta\le1$. The linear maximum principle applied to $u_{\varepsilon,\delta,m}-H$ gives $u_{\varepsilon,\delta,m}\le H$. Similarly, the barrier
	\[
	b_{\varepsilon,m}=C_1(\varepsilon-\rho_m)-C_2(\varepsilon-\rho_m)^2
	\]
	gives a first-order bound at the boundary. For $n\ge2$, the common geometry, zero-order bound and boundary gradient bound therefore allow us to apply \cite[Prop.~3.1]{GuanLi2010TotallyReal}. For $n=1$, the equation is linear and the uniform boundary Schauder estimate gives the same first-order bound. Thus
	\begin{equation}
	\|u_{\varepsilon,\delta,m}\|_{C^{0,1}(\overline{U_{\varepsilon,m}})}\le C.\label{reg:eq-low-order}
	\end{equation}

\begin{lemma}[Uniform collar--core bridge]
\label{reg:lem-uniform-collar-core-bridge}
Assume $n\ge2$.  Let $u_{\varepsilon,\delta,m}\in C^4(\overline{U_{\varepsilon,m}})$ be an admissible solution of
\[
(\omega_0+\ddc u_{\varepsilon,\delta,m})^n =\delta\,\omega_0^n, \qquad u_{\varepsilon,\delta,m}=0 \quad\text{on }\partial U_{\varepsilon,m},
\]
where $\varepsilon\in I$, $0<\delta\le2^{-n}$, and $m$ is sufficiently large, and suppose that
\[
\|u_{\varepsilon,\delta,m}\|_{C^{0,1}(\overline{U_{\varepsilon,m}})}\le G
\]
with $G$ independent of $(\varepsilon,\delta,m)$.  Write $D=U_{\varepsilon,m}$, $u=u_{\varepsilon,\delta,m}$, let $\nu$ be the outward unit normal, and set
\begin{equation}
 M_\delta:=1+\sup_{\partial D}(u_{\nu\nu})^+.
 \label{reg:eq-Mdelta}
\end{equation}
Each $D$ admits a uniformly controlled inward collar function $\varrho$, positive in $D$ and zero on $\partial D$, and there is a fixed collar radius $R_0>0$ with the following property.  For every $0<r_*<R_0$ and every $\eta_0>0$, one can choose $0<r_b<r_*$ and $C_{\eta_0,r_b}>0$, independently of $(\varepsilon,\delta,m)$, such that
\[
 \sup_{\{\varrho\ge r_b\}}|\nabla^2_{g_0}u|_{g_0}  \le\eta_0M_\delta+C_{\eta_0,r_b}.
\]
\end{lemma}

\begin{proof}[Proof outline]
Appendix~\ref{app:subsec-proof-bridge} normalizes the inverse Hessian to have unit trace and applies the maximum principle to a coercive directional quotient. A core estimate then propagates the collar bound to the interior.
\end{proof}

\begin{proposition}[Uniform degenerate boundary Hessian estimate]
\label{reg:prop-uniform-degenerate-boundary}
Let $u_{\varepsilon,\delta,m}\in C^4(\overline{U_{\varepsilon,m}})$ be an admissible solution of
\[
(\omega_0+\ddc u_{\varepsilon,\delta,m})^n =\delta\,\omega_0^n, \qquad u_{\varepsilon,\delta,m}=0 \quad\text{on }\partial U_{\varepsilon,m},
\]
where $\varepsilon\in I$, $0<\delta\le1$, and $m$ is sufficiently large.  If
\[
\|u_{\varepsilon,\delta,m}\|_{C^{0,1}(\overline{U_{\varepsilon,m}})}\le G
\]
with $G$ independent of $(\varepsilon,\delta,m)$, then
\[
\sup_{\partial U_{\varepsilon,m}} |\nabla^2_{g_0}u_{\varepsilon,\delta,m}|_{g_0}\le C,
\]
where $C$ is independent of $(\varepsilon,\delta,m)$. In fact the same constant bounds the real Hessian on the whole domain.  At a boundary point, if
\[
\omega_0+\ddc u_{\varepsilon,\delta,m} =\begin{pmatrix}B&c\\ c^*&d\end{pmatrix}>0
\]
in a unitary frame adapted to the complex tangent space, then also $c^*B^{-1}c\le C$ uniformly.
\end{proposition}

\begin{proof}[Proof outline]
Appendix~\ref{app:subsec-proof-boundary} compares barriers on the physical boundary, lateral face, and inner cap. The collar--core bound controls the cap term, and a fixed-order choice of parameters closes the boundary estimate.
\end{proof}

\begin{proposition}[Stability of smooth-domain Dirichlet approximants]\label{reg:conv-c31-smoothing-reduction}
For fixed $(\varepsilon,\delta)\in I\times(0,1]$, let $u_{\varepsilon,\delta,m}$ be the classical solution of
\[
(\omega_0+\ddc u_{\varepsilon,\delta,m})^n=\delta\,\omega_0^n \quad\text{in }U_{\varepsilon,m}, \qquad u_{\varepsilon,\delta,m}=0 \quad\text{on }\partial U_{\varepsilon,m}.
\]
The gradient and full real-Hessian estimates are uniform in $(\varepsilon,\delta,m)$.  In particular,
\begin{equation}\label{reg:eq-pulled-W2infty}
\sup_{\substack{\varepsilon\in I,\ 0<\delta\le1\\m\gg1}} \|u_{\varepsilon,\delta,m}\circ\Phi_{\varepsilon,m}\|_{W^{2,\infty}(U_\varepsilon)} \le C.
\end{equation}
where $\Phi_{\varepsilon,m}$ is defined in Lemma~\ref{reg:lem-smooth-parallel-geometry}. For every $0<\beta<1$,
\[
u_{\varepsilon,\delta,m}\circ\Phi_{\varepsilon,m} \longrightarrow u_{\varepsilon,\delta} \quad\text{in }C^{1,\beta}(\overline{U_\varepsilon}) \quad\text{as }m\to\infty,
\]
where $u_{\varepsilon,\delta}$ is the Dirichlet solution on $U_\varepsilon$. The convergence is also weakly-* in $W^{2,\infty}(U_\varepsilon)$ after pullback, and
\begin{equation}\label{reg:eq-approximant-C11}
u_{\varepsilon,\delta}\in C^{1,1}(\overline{U_\varepsilon}), \qquad \|u_{\varepsilon,\delta}\|_{C^{1,1}(\overline{U_\varepsilon})}\le C
\end{equation}
uniformly in $(\varepsilon,\delta)$.  On each $K\Subset U_\varepsilon$, the coefficients of the unpulled forms $\ddc u_{\varepsilon,\delta,m}$ converge weakly-* in $L^\infty(K)$, and the mixed Monge--Amp\`ere currents converge in the Bedford--Taylor sense.  For every fixed smooth test form of complementary degree defined near $\overline{U_\varepsilon}$, their pairings over $U_{\varepsilon,m}$ converge to the corresponding pairing over $U_\varepsilon$.
\end{proposition}

\begin{proof}
The low-order bound \eqref{reg:eq-low-order} and Proposition~\ref{reg:prop-uniform-degenerate-boundary} give the uniform real-Hessian estimate. The chain rule for $u_{\varepsilon,\delta,m}\circ\Phi_{\varepsilon,m}$ yields \eqref{reg:eq-pulled-W2infty}.
	
	By \eqref{reg:eq-pulled-W2infty}, the sequence in $m$ extends to a uniformly $C^{1,1}$-bounded family on a common neighborhood. Then we can extract a convergent subsequence by the compact embedding $C^{1,1}\hookrightarrow C^{1,\beta}$ ($\beta<1$). Using $\Phi_{\varepsilon,m}\to\operatorname{Id}$ and the uniform gradient bound, every unpulled subsequential limit agrees on compact subsets with the pulled-back limit. Bedford--Taylor continuity preserves the equation, and the comparison principle identifies this limit as $u_{\varepsilon,\delta}$. Hence
\[
u_{\varepsilon,\delta,m}\circ\Phi_{\varepsilon,m}\longrightarrow u_{\varepsilon,\delta} \quad\text{in }C^{1,\beta}(\overline{U_\varepsilon}).
\]

The estimate of the $W^{2,\infty}$ norm follows from weak-* lower semicontinuity. On each compact subset of $U_\varepsilon$, distributional convergence and the uniform Hessian bound give weak-* convergence of the Hessian coefficients, while Bedford--Taylor continuity gives convergence of the mixed currents. To pass to pairings over the full domains, first remove a fixed thin boundary collar. Its contribution is uniformly small by the coefficient bounds, and $\Vol(U_{\varepsilon,m}\mathbin{\triangle}U_\varepsilon)\to0$. Letting $m\to\infty$ and then shrinking the collar proves the asserted convergence.
\end{proof}

\subsection{Uniform boundary modulus and \texorpdfstring{$L^\infty$}{L-infinity} domain stability}

\begin{convention}[Uniform collar constants]
\label{reg:conv-uniform-collar}
Shrink $\varepsilon_1$ so that the parallel family over $I$ has the uniform geometry of Proposition~\ref{prop:BH_to_collar} and the barrier calculation below is valid in a fixed collar containing $\{|\rho|\le2\varepsilon_1\}$. On that collar put
\[
c_\nabla:=\inf|\nabla\rho|_{\omega_0}>0, \qquad K_1:=\sup|\Delta_{\omega_0}\rho|,
\]
and shrink once more so that $6\varepsilon_1K_1\le c_\nabla^2/2$. Set $M:=\sup_Xu_{\varepsilon_1}$.  Monotonicity gives $0\le u_\varepsilon\le M$ for every $\varepsilon\in I$.
\end{convention}

\begin{lemma}[Uniform boundary modulus]
\label{reg:lem-barrier}
There exists a uniform $L>0$ such that, for every $\varepsilon\in I$,
\begin{equation}\label{reg:eq-uniform-signed-barrier}
0\le u_\varepsilon(x)\le L(\varepsilon-\rho(x))_+
\end{equation}
whenever $x$ belongs to the fixed boundary collar.
\end{lemma}

\begin{proof}
Taking the $\omega_0$-trace of $\omega_0+\ddc u_\varepsilon\ge0$ gives $\Delta_{\omega_0}u_\varepsilon\ge-n$.  On $A_\varepsilon:=\{-2\varepsilon_1<\rho<\varepsilon\}$ put $r=\varepsilon-\rho$ and
\[
h_\varepsilon=C_1r-C_2r^2, \qquad C_2=(n+C_1K_1)c_\nabla^{-2}.
\]
Since $0\le r\le3\varepsilon_1$ on $A_\varepsilon$, direct differentiation and the choice of $\varepsilon_1$ give
\[
\Delta_{\omega_0}h_\varepsilon \le C_1K_1+6\varepsilon_1C_2K_1-2C_2c_\nabla^2\le-n.
\]
With $\varepsilon_1$ and $M$ already fixed in Convention~\ref{reg:conv-uniform-collar}, choose $C_1$ so large that
\[
\frac{3\varepsilon_1n}{C_1c_\nabla^2}\le\frac14, \qquad \frac{C_1\varepsilon_1}{2}\ge M.
\]
Since $3\varepsilon_1K_1/c_\nabla^2\le1/4$, this choice gives $C_1-C_2r\ge C_1/2$ for $0\le r\le3\varepsilon_1$. Uniformly for $\varepsilon\in I$ we therefore have
\[
h_\varepsilon=0\ \text{on }\Sigma_\varepsilon, \qquad h_\varepsilon\ge M\ \text{on }\{\rho=-2\varepsilon_1\},
\]
and $0\le h_\varepsilon\le C_1r$ on $A_\varepsilon$.  The weak maximum principle of Lemma~\ref{prelim:lem-linear-weak-max} gives $u_\varepsilon\le h_\varepsilon$ there.  Outside $U_\varepsilon$ both sides of \eqref{reg:eq-uniform-signed-barrier} vanish.
\end{proof}

Taking $\varepsilon=0$ gives
\[
u_0(x)\le L(-\rho(x))_+
\]
near $\partial U$. Thus the base envelope has boundary value $0$ from the interior side.  The stronger up-to-boundary regularity is obtained below from the fixed-domain Dirichlet theory.

\begin{lemma}[Pairwise $L^\infty$ domain stability]
\label{reg:lem-linf-approx}
There is $C>0$ such that, for all $s,t\in I$,
\begin{equation}
\|u_t-u_s\|_{L^\infty(X)}\le C|t-s|.
\end{equation}
\end{lemma}

\begin{proof}
Assume $s<t$.  Monotonicity gives $u_s\le u_t$.  On $\Omega_s=\{\rho\ge s\}$, Lemma~\ref{reg:lem-barrier} gives
\[
u_t\le L(t-\rho)_+\le L(t-s).
\]
Thus $u_t-L(t-s)$ is admissible for the envelope $u_s$, and hence $0\le u_t-u_s\le L(t-s)$ on $X$.  Interchanging $s$ and $t$ proves the general case.
\end{proof}

\subsection{Fixed-domain Dirichlet theory and the base envelope}

\begin{proposition}[Fixed-domain Dirichlet package]
	\label{reg:prop-boucksom-fixed-package}
	For every $\delta\in(0,1]$, the Dirichlet problem
	\begin{equation}
		\begin{cases}
			(\omega_0+\ddc u_{0,\delta})^n=\delta\,\omega_0^n 			& \text{in }U,\\
			u_{0,\delta}=0 			& \text{on }\partial U
		\end{cases}
	\end{equation}
	admits a unique bounded solution, obtained as the limit of smooth-domain classical solutions, with interior smoothness
	\[
	u_{0,\delta} \in \PSH(U,\omega_0)\cap C^\infty_{\mathrm{loc}}(U).
	\]
		Moreover, there is a constant $C>0$, independent of $\delta$, such that
		\begin{equation}\label{reg:eq-fixed-domain-C11}
		\|u_{0,\delta}\|_{C^{1,1}(\overline U)} 		+\|\Delta_{\omega_0}u_{0,\delta}\|_{L^\infty(U)} \le C ,
		\end{equation}
		and
		\[
		-\omega_0\le \ddc u_{0,\delta}\le C\omega_0 \qquad\text{a.e. on }U .
		\]
	As $\delta\downarrow0$, $u_{0,\delta}$ converges to the unique continuous 	solution $u_D$ of
	\[
	(\omega_0+\ddc u_D)^n=0 \quad\text{in }U, \qquad u_D|_{\partial U}=0.
	\]
		The convergence is strong in $C^{1,\beta}(\overline U)$ for every 		$0<\beta<1$ and weakly-* in $W^{2,\infty}(U)$.
\end{proposition}

\begin{proof}
	For each $\delta>0$, \cite[Ch.~7]{Boucksom2012MABoundary} gives existence and uniqueness for the equation, but the constants in its estimates generally depend on $\delta$.  To obtain \eqref{reg:eq-fixed-domain-C11}, we  specialize Proposition~\ref{reg:conv-c31-smoothing-reduction} to $\varepsilon=0$ and let $m\to\infty$.  We can get the two-sided complex-Hessian estimate from the positivity of $\omega_0+\ddc u_{0,\delta}$ together with its uniform trace bound.  For the convergence to the unique homogeneous solution $u_D$ as $\delta\downarrow0$, the same subsequence argument as before is used, now with the compact embedding $C^{1,\beta'}(\overline U)\hookrightarrow C^{1,\beta}(\overline U)$ ($\beta<\beta'<1$) and the Bedford--Taylor comparison principle.  Finally,  we also need Banach--Alaoglu compactness to prove the weak-* $W^{2,\infty}$ convergence, and the $C^{1,\beta}$ convergence shows that every weak-* cluster point is $u_D$; so we find the whole family converges weak-*.
\end{proof}

\begin{theorem}
	\label{reg:thm-base-regularity}
	The fixed-domain degenerate Dirichlet limit $u_D$ in Proposition~\ref{reg:prop-boucksom-fixed-package} coincides with the restriction of the global equilibrium envelope:
	\[
	u_D=u_0|_U .
	\]
		Consequently,
		\begin{equation}
		u_0\in C^{1,1}(\overline U), \qquad 		\|u_0\|_{C^{1,1}(\overline U)}\le C, 		\qquad -\omega_0\le \ddc u_0\le C\omega_0 		\quad\text{a.e. on }U,
		\end{equation}
		and the Dirichlet approximants $u_{0,\delta}$ satisfy
		\[
		u_{0,\delta}\to u_0 \quad\text{strongly in }C^{1,\beta}(\overline U) 		\quad(0<\beta<1), 		\qquad 		u_{0,\delta}\overset{*}{\rightharpoonup}u_0 		\quad\text{in }W^{2,\infty}(U).
		\]
\end{theorem}

\begin{proof}
	We first prove $u_D=u_0|_U$ by using the supremum in the definition of $u_0$.  By Lemma~\ref{prelim:lem-max-gluing}, the zero extension of $u_D$ is globally $\omega_0$-psh and admissible for $u_0$, so $u_D\le u_0$ on $U$; this gives one inequality.  For the other one, take any function $v$ that is admissible for $u_0$.  Then $v^+=\max\{v,0\}$ vanishes on $X\setminus U$, and its boundary limsup is at most $u_D|_{\partial U}=0$.  Corollary~\ref{prelim:cor-homogeneous-dirichlet-comparison} gives $v\le v^+\le u_D$ on $U$.  Taking the supremum over all such $v$, we obtain $u_0\le u_D$. Proposition~\ref{reg:prop-boucksom-fixed-package} supplies the remaining $C^{1,1}$ estimates and strong convergence; the same weak-* compactness argument gives the asserted $W^{2,\infty}$ convergence.
\end{proof}

\begin{corollary}
	\label{reg:cor-ma-conv}
		For every $0<\beta<1$,
		\[
		u_{0,\delta}\to u_0 \quad\text{strongly in }C^{1,\beta}(\overline U), 		\qquad 		u_{0,\delta}\overset{*}{\rightharpoonup}u_0 		\quad\text{in }W^{2,\infty}(U).
		\]
	Moreover, for every $0\le j\le n$,
	\[
	(\omega_0+\ddc u_{0,\delta})^j\wedge\omega_0^{n-j} \rightharpoonup (\omega_0+\ddc u_0)^j\wedge\omega_0^{n-j}
	\]
	weakly as Radon measures on $U$.  In particular, consistently with Proposition~\ref{prelim:prop-envelope-ma},
	\[
	(\omega_0+\ddc u_0)^n=0 \qquad\text{on }U.
	\]
\end{corollary}

\begin{proof}
The first assertions are contained in Theorem~\ref{reg:thm-base-regularity}; the measure convergence follows from Bedford--Taylor continuity under uniform convergence.  Taking $j=n$ in the approximating equation gives the final assertion.
\end{proof}

\subsection{Uniform regularity on parallel domains}

Proposition~\ref{prop:BH_to_collar} gives the parallel levels a uniform $C^{3,1}$ atlas, a common tubular radius, and a Levi lower bound independent of $\varepsilon\in I$.  Hence Proposition~\ref{reg:conv-c31-smoothing-reduction} applies on every $U_\varepsilon$ with constants independent of $\varepsilon$.  When $n=1$, the same conclusions follow directly from the uniformly controlled linear Dirichlet problem $\ddc u=(\delta-1)\omega_0$.

\begin{proposition}[Uniform moving-domain Dirichlet package]
		\label{reg:lem-uniform-inner-C11}
		There exists a constant $C>0$, independent of
		\[
		(\varepsilon,\delta)\in I\times(0,1],
		\]
	such that, for every such pair $(\varepsilon,\delta)$, the Dirichlet problem
	\begin{equation}\label{reg:eq-approx-moving}
		\begin{cases}
			(\omega_0+\ddc u_{\varepsilon,\delta})^n=\delta\,\omega_0^n 			& \text{in }U_\varepsilon,\\
			u_{\varepsilon,\delta}=0 			& \text{on }\Sigma_\varepsilon
		\end{cases}
	\end{equation}
	admits a unique bounded solution, obtained as the limit of smooth-domain classical solutions, with interior smoothness
	\[
	u_{\varepsilon,\delta}\in \PSH(U_\varepsilon,\omega_0) \cap C^\infty_{\mathrm{loc}}(U_\varepsilon).
	\]
		Moreover,
		\begin{equation}\label{reg:eq-uniform-C11-family}
		\|u_{\varepsilon,\delta}\|_{C^{1,1}(\overline{U_\varepsilon})} 		+\|\Delta_{\omega_0}u_{\varepsilon,\delta}\|_{L^\infty(U_\varepsilon)} 		\le C,
		\end{equation}
		and
		\begin{equation}\label{reg:eq-parallel-uniform-ddc-bounds}
			-\omega_0\le\ddc u_{\varepsilon,\delta}	\le	C\omega_0	\qquad\text{a.e. on }U_\varepsilon .
		\end{equation}
		
		As $\delta\downarrow0$, the functions $u_{\varepsilon,\delta}$ converge to $u_\varepsilon|_{U_\varepsilon}$ strongly in $C^{1,\beta}(\overline{U_\varepsilon})$ for every $0<\beta<1$ and weakly-* in $W^{2,\infty}(U_\varepsilon)$. Consequently,
		\begin{equation}\label{reg:eq-parallel-limit-estimates}
			\|u_\varepsilon\|_{C^{1,1}(\overline{U_\varepsilon})} 			+\|\Delta_{\omega_0}u_\varepsilon\|_{L^\infty(U_\varepsilon)}\le C, 			\qquad -\omega_0\le\ddc u_\varepsilon \le C\omega_0 			\quad\text{a.e. on }U_\varepsilon .
		\end{equation}
\end{proposition}

\begin{proof}
For fixed $(\varepsilon,\delta)$, Proposition~\ref{reg:conv-c31-smoothing-reduction} gives the statement. The compactness and comparison argument of Proposition~\ref{reg:prop-boucksom-fixed-package}, applied on $U_\varepsilon$, gives the limit as $\delta\downarrow0$. Theorem~\ref{reg:thm-base-regularity} identifies the cluster limit.
\end{proof}

\begin{lemma}
	\label{reg:lem-global-lipschitz-moving}
		There exists a constant $C>0$, independent of $\varepsilon\in I$, such that
	\[
	\|u_\varepsilon\|_{C^{0,1}(X)} \le C .
	\]
\end{lemma}

\begin{proof}
	The uniform bound for the gradient $\nabla u_\varepsilon$ is already given by Proposition~\ref{reg:lem-uniform-inner-C11}.  We next show that its distributional gradient is an $L^\infty$ function.  Integrating along length-minimizing geodesics then gives
	\[
	|u_\varepsilon(x)-u_\varepsilon(y)|\le\|\nabla u_\varepsilon\|_{L^\infty}\,d_{g_0}(x,y),
	\]
	and compactness yields the global $C^{0,1}$ bound.  Note that \eqref{reg:eq-uniform-signed-barrier} gives $0\le u_\varepsilon(x)\le L(\varepsilon-\rho(x))_+$, so we can find $u_\varepsilon$ is continuous near $\Sigma_\varepsilon$ and takes values close to $0$ there.  Computing the distributional derivative of its global zero extension directly, with an arbitrary smooth test vector field and in the sense of distributions, we find that the jump across the boundary is $0$; combined with the $L^\infty$ bound for $\nabla u_\varepsilon$, this proves the claim.
\end{proof}

\begin{lemma}
\label{reg:lem-uniform-C11-interpolation}
Let $D$ range over the domains $U_\varepsilon$, $\varepsilon\in I$, or let all functions be pulled back to one of these domains by $\Xi_{s,t}$.  If $f\in C^{1,1}(\overline D)$ and $\|f\|_{C^{1,1}(\overline D)}\le M$, then
\begin{equation}\label{reg:eq-uniform-gradient-interpolation}
\|Df\|_{L^\infty(D)} \le C_M\bigl(\|f\|_{L^\infty(D)}^{1/2} +\|f\|_{L^\infty(D)}\bigr).
\end{equation}
Moreover, for $0<\beta<1$,
\begin{equation}\label{reg:eq-uniform-C11-interpolation}
\|f\|_{C^{1,\beta}(\overline D)} \le C_{M,\beta}\bigl( \|f\|_{L^\infty(D)}+ \|f\|_{L^\infty(D)}^{(1-\beta)/2}\bigr).
\end{equation}
The constants are uniform over the family.
\end{lemma}
\begin{proof}
Set $m=\|f\|_{L^\infty}$. The uniformly controlled boundary charts give extension operators to fixed-size neighborhoods that are bounded both in $L^\infty$ and in $C^{1,1}$, uniformly over the family. Apply Taylor's formula to such an extension $\widetilde f$. For any unit coordinate direction $e$ and $0<h\le h_0$,
\[
|D_ef(x)| \le\frac{|\widetilde f(x+he)|+|\widetilde f(x)|}{h}+CMh \le\frac{Cm}{h}+CMh.
\]
For $m=0$ the claim is immediate. Otherwise choose $h=\min\{h_0,\sqrt{m/(M+1)}\}$. If the second term is the minimum, the bound is $C_Mm^{1/2}$; if $h=h_0$, then $m\ge(M+1)h_0^2$, so the bound is $C_Mm$. This proves \eqref{reg:eq-uniform-gradient-interpolation}. Since $m\le M$, it also gives $\|Df\|_{L^\infty}\le C_Mm^{1/2}$.

For any Lipschitz continuous function $g$ (that is, $g\in C^{0,1}$) and $r=|x-y|$, we have
\[
[g]_{C^{0,\beta}} =\sup_{x\ne y}\frac{|g(x)-g(y)|}{|x-y|^\beta} \le 2^{1-\beta}\|g\|_{L^\infty}^{1-\beta}[g]_{C^{0,1}}^{\beta} \le C\|g\|_{L^\infty}^{1-\beta}[g]_{C^{0,1}}^{\beta},
\]
where we used the elementary inequality
\[
|g(x)-g(y)| \le\min\bigl(2\|g\|_{L^\infty},\,[g]_{C^{0,1}}r\bigr) \le\bigl(2\|g\|_{L^\infty}\bigr)^{1-\beta}\bigl([g]_{C^{0,1}}r\bigr)^{\beta}.
\]
Taking $g=Df$, we obtain
\[
[Df]_{C^{0,\beta}} \le C\bigl(C_Mm^{1/2}\bigr)^{1-\beta}M^{\beta}.
\]
Now
\[
\|f\|_{C^{1,\beta}} =\underbrace{\|f\|_{L^\infty}}_{m} +\underbrace{\|Df\|_{L^\infty}}_{\le C_M(m^{1/2}+m)} +\underbrace{[Df]_{C^{0,\beta}}}_{\le C_{M,\beta}m^{(1-\beta)/2}}.
\]
Since the term $m^{1/2}$ is absorbed by $m^{(1-\beta)/2}$ or by $m$,
\[
\|f\|_{C^{1,\beta}} \le C_{M,\beta}\bigl(m+m^{(1-\beta)/2}\bigr),
\]
which is \eqref{reg:eq-uniform-C11-interpolation}.
\end{proof}

\begin{proposition}
\label{reg:prop-quantitative-delta-domain-stability}

For every $0<\beta<1$ there are constants $C,C_\beta>0$ such that, for $\varepsilon\in I$ and $0<\delta\le1$,
\begin{align}
\|u_{\varepsilon,\delta}-u_\varepsilon\|_{L^\infty(U_\varepsilon)} &\le C\delta^{1/n},
\label{reg:eq-delta-Linfty-rate}\\
\|u_{\varepsilon,\delta}-u_\varepsilon\|_{C^{1,\beta}(\overline{U_\varepsilon})} &\le C_\beta\delta^{(1-\beta)/(2n)}.
\label{reg:eq-delta-C1beta-rate}
\end{align}
For every $s,t\in I$,
\begin{align}
\|u_t\circ\Xi_{s,t}-u_s\|_{L^\infty(U_s)} &\le C|t-s|, \\
\|u_t\circ\Xi_{s,t}-u_s\|_{C^{1,\beta}(\overline{U_s})} &\le C_\beta|t-s|^{(1-\beta)/2}.
\label{reg:eq-t-pullback-C1beta-rate}
\end{align}
If $s\le t$, so that $U_s\subset U_t$, then also
\begin{equation}\label{reg:eq-t-common-domain-C1beta-rate}
\|u_t-u_s\|_{C^{1,\beta}(\overline{U_s})} \le C_\beta|t-s|^{(1-\beta)/2}.
\end{equation}
In particular,
\[
\|\nabla(u_{\varepsilon,\delta}-u_\varepsilon)\|_{L^\infty(U_\varepsilon)} \le C\delta^{1/(2n)}, \qquad \|\nabla(u_t\circ\Xi_{s,t}-u_s)\|_{L^\infty(U_s)} \le C|t-s|^{1/2}.
\]
If $\nu_t:=\nabla_{g_0}\rho|_{\Sigma_t}$ and $q_t:=-\partial_{\nu_t}u_t$, then
\begin{equation}\label{reg:eq-normal-derivative-half-holder}
\|q_t\circ\Psi_{s,t}-q_s\|_{C^0(\Sigma_s)} \le C|t-s|^{1/2}.
\end{equation}
\end{proposition}

\begin{proof}
Put $a=\delta^{1/n}$. Since
\[
\omega_0+\ddc((1-a)u_\varepsilon) = a\omega_0+(1-a)(\omega_0+\ddc u_\varepsilon),
\]
positivity and the binomial expansion give
\[
(\omega_0+\ddc((1-a)u_\varepsilon))^n \ge a^n\omega_0^n = \delta\omega_0^n.
\]
The comparison principle, applied with the common zero boundary value, gives
\[
(1-a)u_\varepsilon \le u_{\varepsilon,\delta} \le u_\varepsilon.
\]
The uniform $L^\infty$ bound proves \eqref{reg:eq-delta-Linfty-rate}.

We can use Lemma~\ref{reg:lem-linf-approx}, Lemma~\ref{reg:lem-global-lipschitz-moving}, and \eqref{prelim:eq-parallel-identification-bounds} to obtain the estimate for the domain parameter:
\[
\begin{aligned}
	\|u_t\circ\Xi_{s,t}-u_s\|_{L^\infty(U_s)} 	&\le \|u_t\circ\Xi_{s,t}-u_t\|_{L^\infty(U_s)} + \|u_t-u_s\|_{L^\infty(X)} \\
	&\le C|t-s|.
\end{aligned}
\]
The differences in these two comparisons have uniformly bounded $C^{1,1}$ norms; for the second one this follows from the chain rule and the uniform $C^2$ bounds for $\Xi_{s,t}$. Lemma~\ref{reg:lem-uniform-C11-interpolation} therefore proves \eqref{reg:eq-delta-C1beta-rate} and \eqref{reg:eq-t-pullback-C1beta-rate}, including the displayed gradient rates. When $s\le t$, apply the same interpolation lemma directly to $u_t-u_s$ on the common domain $U_s$ to obtain \eqref{reg:eq-t-common-domain-C1beta-rate}.

On $\Sigma_s$ put
\[
Z_{s,t} := (D\Xi_{s,t})^{-1}(\nu_t\circ\Psi_{s,t}), \qquad w_{s,t} := u_t\circ\Xi_{s,t}-u_s.
\]
The chain rule gives the exact identity
\[
q_t\circ\Psi_{s,t}-q_s = -dw_{s,t}(Z_{s,t}) + du_s(\nu_s-Z_{s,t}).
\]
The uniform $C^2$ estimate for the normal-flow identifications implies
\[
\|Z_{s,t}-\nu_s\|_{C^0(\Sigma_s)} \le C|t-s|, \qquad \|Z_{s,t}\|_{C^0(\Sigma_s)} \le C.
\]
Finally, after combining these inequalities with \eqref{reg:eq-uniform-gradient-interpolation} and the uniform gradient bound we can prove \eqref{reg:eq-normal-derivative-half-holder}.
\end{proof}

\section{Flux, trace, and singular boundary mass}\label{sec:flux}

We identify $\mu_0^{\rm sing}:=\MA(u_0)\llcorner\partial U$ with the negative outward weak normal trace of the envelope flux and obtain the boundary densities used in Section~\ref{sec:hadamard}. The Hodge identification in Lemma~\ref{lem:hodge} places the bounded flux forms in the $\DM^\infty$ framework of Subsection~\ref{prelim:subsec-dm}.

\subsection{Flux current and divergence identity}

For a sufficiently regular function $u$, recall the flux current
\[
T_u:=\sum_{j=0}^{n-1}\omega_u^j\wedge\omega_0^{n-1-j}, \qquad \mathcal J[u]:=\frac1V\dc u\wedge T_u.
\]
Theorem~\ref{thm:global-trace} gives its outward trace whenever $\mathcal J[u]\in\DM^\infty(U)$.

\begin{corollary}[Divergence and base-flux trace recovery]

\label{stat:cor-base-flux-trace-recovery}
For the base envelope, $\mathcal J[u_0]\in\DM^\infty(U)$,
\[
d\mathcal J[u_0]=-\frac1V\omega_0^n\quad\text{on }U.
\]
Its outward trace is recovered from the inner parallel levels in the sense of Theorem~\ref{prelim:thm-innerparallel}.
\end{corollary}
\begin{proof}
	
	We note that for smooth $u$, factorization and closedness directly give
	\[
	\omega_u^n-\omega_0^n=\ddc u\wedge\sum_{j=0}^{n-1}\omega_u^j\wedge\omega_0^{n-1-j}=d\!\left(\dc u\wedge\sum_{j=0}^{n-1}\omega_u^j\wedge\omega_0^{n-1-j}\right),
	\]
	so $d\mathcal J[u]=V^{-1}(\omega_u^n-\omega_0^n)$.
	We now approximate the desired conclusion by the smooth case. Applying this identity to $u_{0,\delta,m}$ on $U_{0,m}$ gives
	\[
	d\mathcal J[u_{0,\delta,m}]=\frac{\delta-1}{V}\omega_0^n.
	\]
	Proposition~\ref{reg:conv-c31-smoothing-reduction} allows us to let $m\to\infty$, so that on $U$ we have
	\[
	d\mathcal J[u_{0,\delta}]=\frac1V\bigl(\omega_{u_{0,\delta}}^n-\omega_0^n\bigr)=\frac{\delta-1}{V}\omega_0^n.
	\]
	Set $T_\delta:=T_{u_{0,\delta}}$ and $T_0:=T_{u_0}$, and let $\eta$ be a smooth compactly supported real $1$-form on $U$.
	Adding and subtracting $\langle T_\delta,\eta\wedge\dc u_0\rangle$ gives
	\begin{align*}
		V\langle\mathcal J[u_{0,\delta}]-\mathcal J[u_0],\eta\rangle
		&=\langle T_\delta,\eta\wedge\dc(u_{0,\delta}-u_0)\rangle+\langle T_\delta-T_0,\eta\wedge\dc u_0\rangle.
	\end{align*}
	The first term satisfies
	\[
	\bigl\vert{}\langle T_\delta,\eta\wedge\dc(u_{0,\delta}-u_0)\rangle\bigr\vert{}\le C_\eta\,\Vert{}T_\delta\Vert{}(U)\,\Vert{}u_{0,\delta}-u_0\Vert{}_{C^1(\overline U)}\longrightarrow0.
	\]
	For the second term, set $\psi:=\eta\wedge\dc u_0$ and choose smooth compactly supported $2$-forms $\psi_k$ with $\psi_k\to\psi$ uniformly.
	Then
	\[
	\bigl\vert{}\langle T_\delta-T_0,\psi\rangle\bigr\vert{}\le\bigl\vert{}\langle T_\delta-T_0,\psi_k\rangle\bigr\vert{}+\bigl(\Vert{}T_\delta\Vert{}(U)+\Vert{}T_0\Vert{}(U)\bigr)\Vert{}\psi-\psi_k\Vert{}_{C^0}.
	\]
	For fixed $k$, the first term on the right tends to zero by the weak convergence of currents, while the second term becomes uniformly small as $k\to\infty$.
	Therefore,
	\[
	\mathcal J[u_{0,\delta}]\rightharpoonup\mathcal J[u_0]\qquad\text{as currents on }U.
	\]

\end{proof}

We record two continuity statements for the fluxes used below.
\begin{lemma}[Strong--weak continuity of flux products]\label{stat:lem-strong-weak-flux}
Let $B_k\rightharpoonup B$ be order-zero currents of a fixed degree on a domain $D$, with locally uniformly bounded mass.
Let $a_k\to a$ locally uniformly be continuous forms of a fixed degree such that $a_k\wedge B_k$ is defined.
Then
\[
a_k\wedge B_k\rightharpoonup a\wedge B\qquad\text{as currents on }D.
\]
\end{lemma}
\begin{proof}
Write
\[
a_k\wedge B_k-a\wedge B=(a_k-a)\wedge B_k+a\wedge(B_k-B).
\]
When tested against a smooth compactly supported form, the first term tends to zero by uniform convergence and the mass bound.
For the second term, approximate the continuous test form containing $a$ uniformly by smooth forms with a common compact support.
Weak current convergence and the same mass bound give the limit.
\end{proof}

\begin{lemma}[Trace stability]\label{dyn:lem-trace-stability}
Let $J_\varepsilon,J_0\in\DM^\infty(U)$ and suppose that, on every $U_{-r}\Subset U$, $J_\varepsilon\rightharpoonup J_0$ as currents and $dJ_\varepsilon\rightharpoonup^*dJ_0$ as Radon measures.
If the coefficients of $J_\varepsilon$ and the densities of $dJ_\varepsilon$ are uniformly bounded in a fixed boundary collar for $0\le\varepsilon\le\bar\varepsilon$, then
\[
\Tr_{\partial U}^{\mathrm{out}}J_\varepsilon\rightharpoonup^*\Tr_{\partial U}^{\mathrm{out}}J_0\qquad\text{in }\mathcal M(\partial U).
\]
\end{lemma}
\begin{proof}
For $\Phi$ smooth near $\overline U$, put $\varphi=\Phi|_{\partial U}$. The Gauss--Green identity gives
\[
I_\varepsilon(\varphi):=\int_{\partial U}\varphi\,d(\Tr_{\partial U}^{\mathrm{out}}J_\varepsilon)=\int_U\Phi\,d(dJ_\varepsilon)+\int_Ud\Phi\wedge J_\varepsilon.
\]
Choose $\chi_r\in C_c^\infty(U)$ with $0\le\chi_r\le1$, equal to $1$ on $U_{-2r}$ and supported in $U_{-r}$. Split both integrals using $\chi_r+(1-\chi_r)$.

For fixed $r$, the terms containing $\chi_r$ converge by the interior assumptions. The collar terms are bounded by $Cr\|\Phi\|_{C^1(\overline U)}$, so
\[
\limsup_{\varepsilon\downarrow0}|I_\varepsilon(\varphi)-I_0(\varphi)|\le2Cr\|\Phi\|_{C^1(\overline U)}.
\]
Letting $r\downarrow0$ proves convergence for smooth boundary tests. A fixed collar cutoff and Theorem~\ref{thm:global-trace} give a uniform total-variation bound. Density then extends the convergence to every test in $C(\partial U)$.
\end{proof}

\subsection{Absolute continuity and jump density}
The next lemma isolates the boundary relation used for all choices of weights.

\begin{lemma}[Boundary proportionality of flux traces]\label{stat:lem-boundary-flux-proportionality}
Let $B$ be a closed real $(n-1,n-1)$-current with $L^\infty$ coefficients on $U$, and let $v\in C^{1,1}(\overline U)$ vanish on $\Sigma$.
Put $J_g:=V^{-1}\dc g\wedge B$ for $g=\rho,v$.
Then $J_\rho,J_v\in\DM^\infty(U)$ and
\[
\Tr_\Sigma^{\mathrm{out}}J_v=\partial_{\nu_{\mathrm{out}}}v\,\Tr_\Sigma^{\mathrm{out}}J_\rho.
\]
\end{lemma}
\begin{proof}
Closedness of $B$ and the weak product rule give $dJ_g=V^{-1}\ddc g\wedge B\in L^\infty(U)$, so both currents lie in $\DM^\infty(U)$.
Set $f:=\partial_{\nu_{\mathrm{out}}}v$ and choose a Lipschitz extension $\widetilde f$ to $\overline U$ that is constant along normal segments in a boundary collar.
Since $dv=f\,d\rho$ on $\Sigma$, the $C^{1,1}$ bound and the normal flow give
\begin{equation}\label{stat:eq-off-boundary-dc-estimate}
\|\iota_{-r}^*(\dc v-\widetilde f\,\dc\rho)\|_{L^\infty(\Sigma_{-r})}\le Cr.
\end{equation}
Here $\iota_{-r}:\Sigma_{-r}\hookrightarrow U$ is the inclusion.
By the product rule, $E:=J_v-\widetilde fJ_\rho\in\DM^\infty(U)$. The coefficient bound for $B$ gives $\|\iota_{-r}^*E\|_{L^\infty}\le Cr\|B\|_{L^\infty}$ for almost every small $r>0$.
Inner-parallel recovery in Theorem~\ref{thm:global-trace} and the uniform collar area bound give $\Tr_\Sigma^{\mathrm{out}}E=0$.
The trace product rule now proves the identity.
\end{proof}

For nonnegative weights $c_p$, the next proposition gives the trace relation between the physical and geometric reference flux currents.

\begin{proposition}[Weighted absolute continuity and RN density]\label{prop:jump-density}
Let $c_0,\ldots,c_{n-1}\ge0$ satisfy $\sum_pc_p>0$, and set
\[
B_c:=\sum_{p=0}^{n-1}c_p\,\omega_{u_0}^p\wedge\omega_0^{n-1-p},\quad \mathcal J[\rho;c]:=\frac1V\dc\rho\wedge B_c,\quad \mathcal J[u_0;c]:=\frac1V\dc u_0\wedge B_c.
\]
Then $\mathcal J[\rho;c],\mathcal J[u_0;c]\in\DM^\infty(U)$, $d\sigma_{\rho,c}:=\Tr_\Sigma^{\mathrm{out}}\mathcal J[\rho;c]$ is a positive Radon measure, and
\[
\Tr_\Sigma^{\mathrm{out}}\mathcal J[u_0;c] =\partial_{\nu_{\mathrm{out}}}u_0\,d\sigma_{\rho,c}\le0.
\]
\end{proposition}

\begin{proof}
The current $B_c$ is positive, closed, and has $L^\infty$ coefficients.
Lemma~\ref{stat:lem-boundary-flux-proportionality} gives the $\DM^\infty$ membership and the Radon--Nikodym identity.
For positivity, write $\mathcal J[\rho;c]=\iota_FdV_{g_0}$ in the signed-distance collar.
Then
\[
(F\cdot\nabla\rho)\,dV_{g_0}=d\rho\wedge\mathcal J[\rho;c]=\frac1Vd\rho\wedge\dc\rho\wedge B_c\ge0\qquad\text{a.e.}
\]
By coarea, the outward flux is nonnegative on almost every inner parallel level. Theorem~\ref{thm:global-trace} therefore gives a nonnegative trace.
Thus $d\sigma_{\rho,c}\ge0$. Since $u_0\ge0$ on $U$ and $u_0|_\Sigma=0$, we also have $\partial_{\nu_{\mathrm{out}}}u_0\le0$.
\end{proof}

\subsection{Boundary mass as outward flux}

The following lemma combines zero-boundary integration by parts with Gauss--Green extraction; it also applies to the mixed currents in Section~\ref{sec:hadamard}.
\begin{lemma}[Boundary extraction by integration by parts]
\label{stat:lem-bt-gg-extraction}
Let $v\in C^{1,1}(\overline U)$ vanish on $\Sigma$, let $B$ be a closed real $(n-1,n-1)$-current with $L^\infty$ coefficients on $U$, and put $J:=V^{-1}\dc v\wedge B$.
Suppose that $\nu$ is a finite signed Radon measure on $X$, supported in $\overline U$, and that
\[
\int_X\phi\,d\nu=\frac1V\int_Uv\,\ddc\phi\wedge B\qquad(\phi\in C^\infty(X)).
\]
Then $\nu\llcorner U=dJ$ and $\nu\llcorner\Sigma=-\Tr_\Sigma^{\mathrm{out}}J$.
\end{lemma}
\begin{proof}
By Lemma~\ref{stat:lem-boundary-flux-proportionality}, we have $J\in\DM^\infty(U)$.
For smooth $\phi$, set $K_\phi:=V^{-1}\dc\phi\wedge B$.
The two facts $dK_\phi=V^{-1}\ddc\phi\wedge B\in L^\infty(U)$ and $v|_\Sigma=0$ allow us to use the trace product rule, giving $\Tr_\Sigma^{\mathrm{out}}(vK_\phi)=0$.
Gauss--Green then gives
\[
\frac1V\int_Uv\,\ddc\phi\wedge B=-\frac1V\int_Udv\wedge\dc\phi\wedge B=-\int_Ud\phi\wedge J.
\]
Here the second equality uses the $(n-1,n-1)$ type of $B$.
Testing with compact support in $U$ shows that $\nu\llcorner U=dJ$.
Split $\nu$ over $U$ and $\Sigma$, and use Gauss--Green again:
\[
\int_\Sigma\phi\,d\nu=\int_X\phi\,d\nu-\int_U\phi\,d(dJ)=-\int_\Sigma\phi\,d(\Tr_\Sigma^{\mathrm{out}}J).
\]
Restrictions of ambient smooth functions are dense in $C(\Sigma)$, which proves the boundary identity.
\end{proof}

\begin{proposition}[Outward flux identifies the boundary mass]\label{stat:prop-boundary-mass}
	On the prescribed interface,
	\begin{equation}\label{stat:eq-boundary-mass-identity}
		\mu_0^{\mathrm{sing}} 		=-\Tr^{\mathrm{out}}_{\partial U}(\mathcal J[u_0]) 		=\left(-\frac{\partial u_0}{\partial\nu_{\mathrm{out}}}\right)d\sigma_\rho 		\qquad\text{in }\mathcal M(\partial U),
	\end{equation}
	where $\mu_0^{\mathrm{sing}}:=\mu_0\llcorner\partial U$ and $\mu_0=\MA(u_0)$.
\end{proposition}

\begin{proof}
Take $\phi\in C^\infty(X)$.
Global Bedford--Taylor integration by parts, using $u_0=0$ on $X\setminus U$, gives
\begin{equation}\label{stat:eq-global-bt-ibp}
\int_X\phi\left(\mu_0-\frac1V\omega_0^n\right)=\frac1V\int_Uu_0\,\ddc\phi\wedge T_{u_0}.
\end{equation}
Put $\nu:=\mu_0-V^{-1}\omega_0^n$.
The equality $\mu_0=V^{-1}\omega_0^n$ on $X\setminus\overline U$ shows that this measure is supported in $\overline U$, with $\nu\llcorner\Sigma=\mu_0^{\mathrm{sing}}$.
Apply Lemma~\ref{stat:lem-bt-gg-extraction} with $v=u_0|_U$ and $B=T_{u_0}|_U$ to obtain the first equality in \eqref{stat:eq-boundary-mass-identity}.
The second follows from Proposition~\ref{prop:jump-density}, completing the proof of Theorem~\ref{intro:thm-boundary-mass}.
\end{proof}

\subsection{Total geometric jump density}
For $n\ge2$, polarization produces mixed boundary terms with weights $p+1$. The following current packages them; Example~\ref{dyn:ex-toric-collapse} computes the contribution of these weights explicitly.

\begin{proposition}[Total geometric jump density]\label{prop:total_jump_density}
With $B_0^{\mathrm{tot}}$ as in Section~\ref{sec:intro}, set
\[
\mathcal J[\rho;\mathrm{tot}]:=\frac1V\dc\rho\wedge B_0^{\mathrm{tot}}, \qquad \mathcal J[u_0;\mathrm{tot}]:=\frac1V\dc u_0\wedge B_0^{\mathrm{tot}}.
\]
Both currents lie in $\DM^\infty(U)$, and
\begin{equation}\label{stat:eq-total-jump-density}
d\sigma_\rho^{\mathrm{tot}} :=\Tr_\Sigma^{\mathrm{out}}\mathcal J[\rho;\mathrm{tot}]\ge0, \qquad \sigma_{\mathrm{tot}} :=-\Tr_\Sigma^{\mathrm{out}}\mathcal J[u_0;\mathrm{tot}] =\left(-\partial_{\nu_{\mathrm{out}}}u_0\right)d\sigma_\rho^{\mathrm{tot}}.
\end{equation}
\end{proposition}

\begin{proof}
	Apply Proposition~\ref{prop:jump-density} with $c_p=p+1$.
\end{proof}

\begin{proposition}[Explicit densities on the parallel levels]
\label{stat:prop-explicit-parallel-density}
For $t\in I$, let $\iota_t:\Sigma_t\hookrightarrow X$ be the inclusion and orient $\Sigma_t=\partial U_t$ by the outward unit normal $\nu_t=\nabla_{g_0}\rho$.
Put
\[
\theta_t:=\iota_t^*\dc\rho,\qquad \alpha_t:=\iota_t^*\omega_0,\qquad q_t:=-\partial_{\nu_t}u_t,\qquad A_t:=\alpha_t-q_t\,d\theta_t.
\]
Let $d\sigma_{\rho,t}$ and $d\sigma_{\rho,t}^{\mathrm{tot}}$ be the recentered traces from Proposition~\ref{prop:jump-density} for $c_p=1$ and $c_p=p+1$.
Then $A_t$ is semipositive on $\ker\theta_t$, and
\begin{align}
d\sigma_{\rho,t}&=\frac1V\theta_t\wedge\sum_{p=0}^{n-1}A_t^p\wedge\alpha_t^{n-1-p},\label{stat:eq-explicit-static-density}\\
d\sigma_{\rho,t}^{\mathrm{tot}}&=\frac1V\theta_t\wedge\sum_{p=0}^{n-1}(p+1)A_t^p\wedge\alpha_t^{n-1-p}.\label{stat:eq-explicit-total-density}
\end{align}
These positive measures are represented by Lipschitz top-degree forms.
With $dS_t$ denoting the $g_0$-surface measure, there are $c,C>0$, independent of $t\in I$, such that
\begin{equation}\label{stat:eq-reference-surface-equivalence}
c\,dS_t\le d\sigma_{\rho,t}\le C\,dS_t,\qquad c\,dS_t\le d\sigma_{\rho,t}^{\mathrm{tot}}\le C\,dS_t.
\end{equation}
\end{proposition}
\begin{proof}
\noindent\textbf{Step 1. The smooth boundary identity.}
Fix $t\in I$ and set
\[
D_m:=U_{t,m},\qquad \Sigma_{t,m}:=\partial D_m,\qquad \nu_{t,m}:=\frac{\nabla\rho_m}{|d\rho_m|_{g_0}},\qquad a_{\delta,m}:=-\frac{\partial_{\nu_{t,m}}u_{t,\delta,m}}{|d\rho_m|_{g_0}}.
\]
Let $\iota_{t,m}:\Sigma_{t,m}\hookrightarrow X$.
Set $\theta_{t,m}:=\iota_{t,m}^*\dc\rho_m$, $\alpha_{t,m}:=\iota_{t,m}^*\omega_0$, and $A_{\delta,m}:=\alpha_{t,m}-a_{\delta,m}d\theta_{t,m}$.
Since $u_{t,\delta,m}=0$ on $\Sigma_{t,m}$,
\[
\iota_{t,m}^*\dc u_{t,\delta,m}=-a_{\delta,m}\theta_{t,m},\qquad \iota_{t,m}^*\ddc u_{t,\delta,m}=-da_{\delta,m}\wedge\theta_{t,m}-a_{\delta,m}d\theta_{t,m}.
\]
Wedge with $\theta_{t,m}$. The $da_{\delta,m}\wedge\theta_{t,m}$ terms vanish, giving
\begin{equation}\label{stat:eq-boundary-da-cancellation}
\theta_{t,m}\wedge\iota_{t,m}^*(\omega_{u_{t,\delta,m}}^p\wedge\omega_0^{n-1-p})=\theta_{t,m}\wedge A_{\delta,m}^p\wedge\alpha_{t,m}^{n-1-p}.
\end{equation}

\noindent\textbf{Step 2. Identification of the weak traces.}
For $w_p=1$ or $w_p=p+1$, set
\[
B_{\delta,m}^{(w)}:=\sum_{p=0}^{n-1}w_p\omega_{u_{t,\delta,m}}^p\wedge\omega_0^{n-1-p},\qquad J_{\delta,m}^{(w)}:=\frac1V\dc\rho_m\wedge B_{\delta,m}^{(w)}.
\]
Choose a smooth ambient test function $\widetilde\varphi$ supported in $N_{2\eta}$.
Stokes' theorem gives
\begin{align}
\int_{\Sigma_{t,m}}\widetilde\varphi\,\iota_{t,m}^*J_{\delta,m}^{(w)}&=\frac1V\int_{D_m\cap N_{2\eta}}\widetilde\varphi\,\ddc\rho_m\wedge B_{\delta,m}^{(w)}\notag\\
&\quad+\int_{D_m\cap N_{2\eta}}d\widetilde\varphi\wedge J_{\delta,m}^{(w)}.\label{stat:eq-localized-smooth-stokes}
\end{align}
For fixed $\delta$, Proposition~\ref{reg:conv-c31-smoothing-reduction}, Lemma~\ref{stat:lem-strong-weak-flux}, and $\rho_m\to\rho$ in $C^2$ give convergence of the right-hand side as $m\to\infty$. Uniform coefficient bounds control the thin-collar and domain-symmetric-difference errors. The boundary $C^1$ convergence gives, after identification,
\[
a_{\delta,m}\longrightarrow q_{t,\delta}:=-\partial_{\nu_t}u_{t,\delta},\qquad (\theta_{t,m},d\theta_{t,m},\alpha_{t,m})\longrightarrow(\theta_t,d\theta_t,\alpha_t)
\]
uniformly. Hence the boundary polynomials in \eqref{stat:eq-boundary-da-cancellation} converge uniformly, and \eqref{stat:eq-localized-smooth-stokes} identifies their limits with the Gauss--Green traces on $U_t$.

Next let $\delta\downarrow0$. Proposition~\ref{reg:lem-uniform-inner-C11} gives $q_{t,\delta}\to q_t$ uniformly. Bedford--Taylor convergence and the uniform coefficient bounds give weak-* convergence of the mixed coefficients in $L^\infty$ on compact subsets. Multiplication by the fixed forms $\dc\rho$ and $\ddc\rho$ gives convergence of the reference fluxes and their divergences. Lemma~\ref{dyn:lem-trace-stability} then yields weak-* convergence of their boundary traces. The density polynomials converge uniformly and give \eqref{stat:eq-explicit-static-density}--\eqref{stat:eq-explicit-total-density}.

The restriction $A_{\delta,m}$ of $\omega_{u_{t,\delta,m}}$ to the complex tangent space is semipositive. Its uniform boundary limit $A_t$ is semipositive on $\ker\theta_t$, so each density term is nonnegative. The $p=0$ term satisfies
\[
\theta_t\wedge\alpha_t^{n-1}=\frac1{2n}\iota_t^*(\iota_{\nu_t}\omega_0^n)=\frac{(n-1)!}{2}\,dS_t.
\]
The uniform $C^{1,1}$ estimate and collar geometry bound the density polynomials and their Lipschitz norms.
\end{proof}

\begin{corollary}[Nondegeneracy and stability of the parallel densities]\label{stat:cor-parallel-density-bounds}
With the notation above, there are $c,C>0$, independent of $t\in I$, such that $c\le q_t\le C$ on $\Sigma_t$.
Put $\mu_t^{\mathrm{sing}}:=\MA(u_t)\llcorner\Sigma_t$ and $\sigma_{\mathrm{tot},t}:=-\Tr_{\Sigma_t}^{\mathrm{out}}(V^{-1}\dc u_t\wedge B_t^{\mathrm{tot}})$.
Then
\begin{equation}\label{stat:eq-explicit-physical-densities}
\mu_t^{\mathrm{sing}}=q_t\,d\sigma_{\rho,t},\qquad \sigma_{\mathrm{tot},t}=q_t\,d\sigma_{\rho,t}^{\mathrm{tot}}.
\end{equation}
These measures and $q_t^2d\sigma_{\rho,t}^{\mathrm{tot}}$ are uniformly equivalent to $dS_t$ and have full support.
For $s,t\in I$, the boundary identification in Lemma~\ref{prelim:lem-parallel-identifications} satisfies
\begin{equation}\label{stat:eq-total-density-half-holder}
\left\|\Psi_{s,t}^*d\sigma_{\rho,t}^{\mathrm{tot}}-d\sigma_{\rho,s}^{\mathrm{tot}}\right\|_{C^0(\Sigma_s,g_0)}\le C|t-s|^{1/2},
\end{equation}
where the norm is the coefficient norm of the represented top-degree forms.
\end{corollary}
\begin{proof}
Choose $\gamma>0$ with $\omega_0-\gamma\ddc\rho\ge0$.
The admissible competitor $\gamma(t-\rho)_+\le u_t$ gives $q_t\ge\gamma$. The gradient estimate bounds $q_t$ from above.
Recentering Propositions~\ref{stat:prop-boundary-mass} and~\ref{prop:total_jump_density} gives \eqref{stat:eq-explicit-physical-densities}.
Together with \eqref{stat:eq-reference-surface-equivalence}, these bounds give equivalence to surface measure and full support.

Under $\Psi_{s,t}$, the forms $\theta_t,d\theta_t,\alpha_t$ vary by $O(|t-s|)$, while \eqref{reg:eq-normal-derivative-half-holder} gives $q_t\circ\Psi_{s,t}-q_s=O(|t-s|^{1/2})$.
Applying these estimates to the finite polynomial \eqref{stat:eq-explicit-total-density} proves \eqref{stat:eq-total-density-half-holder}.
\end{proof}

\section{First variation of the energy}\label{sec:hadamard}
The finite-difference identity for the Monge--Amp\`ere energy \cite[Corollary~4.2]{BermanBoucksom2010EnergyEquilibrium} localizes the increment to mixed boundary measures. Their weak normal traces then yield the Hadamard formula along the signed parallel family.

\subsection{Boundary localization}\label{dyn:subsec-localization}
\begin{lemma}\label{dyn:lem-first-order-expansion}
As $\varepsilon\downarrow0$,
\[
u_\varepsilon(x)-u_0(x)=-\varepsilon\partial_{\nu_{\mathrm{out}}}u_0(x)+o(\varepsilon)
\]
uniformly for $x\in\partial U$.
\end{lemma}
\begin{proof}
Fix $x\in\partial U$ and put $x_s=\exp_x(s\nu_{\mathrm{out}})$ and $\nu_s=\nabla_{g_0}\rho(x_s)$.
Since $x_\varepsilon\in\Sigma_\varepsilon$ and $u_\varepsilon(x_\varepsilon)=u_0(x)=0$,
\[
u_\varepsilon(x)-u_0(x) =-\int_0^\varepsilon\partial_{\nu_s}u_\varepsilon(x_s)\,ds.
\]
The uniform $C^{1,\beta}$ estimate, together with the $C^1$ normal field, gives $\partial_{\nu_s}u_\varepsilon(x_s)=\partial_{\nu_{\mathrm{out}}}u_\varepsilon(x)+O(s^\beta)$ uniformly.
Equation~\eqref{reg:eq-t-common-domain-C1beta-rate} also gives $\partial_{\nu_{\mathrm{out}}}u_\varepsilon\to\partial_{\nu_{\mathrm{out}}}u_0$ uniformly on $\partial U$.
Insert these estimates in the integral to conclude.
\end{proof}

By Lemma~\ref{prelim:lem-am-telescoping},
\begin{equation}\label{dyn:eq-energy-algebraic}
E_{\omega_0}(u_\varepsilon)-E_{\omega_0}(u_0)=\frac1{n+1}\sum_{j=0}^{n}\int_X(u_\varepsilon-u_0)d\mu_{0,\varepsilon}^{(j)}, \quad \mu_{0,\varepsilon}^{(j)}:=\frac1V\omega_{u_\varepsilon}^{j}\wedge\omega_{u_0}^{n-j}.
\end{equation}
\begin{lemma}\label{dyn:lem-bulk-shell-vanishing}
For every $0\le j\le n$,
\[
\int_{X\setminus\partial U}(u_\varepsilon-u_0)d\mu_{0,\varepsilon}^{(j)}=o(\varepsilon).
\]
\end{lemma}
\begin{proof}
	
	Note that
	\[
	X\setminus\partial U=(X\setminus U_\varepsilon)\cup\{0<\rho<\varepsilon\}\cup U.
	\]
	On the first set the integrand is always $0$.  On the second set, after using these four facts: $u_0=0$, $0\le u_\varepsilon\le C\varepsilon$, and the mixed densities are uniformly bounded by Theorem~\ref{intro:thm-uniform-regularity}, and the shell has volume $O(\varepsilon)$ by Proposition~\ref{prop:BH_to_collar}, we will find this part is $O(\varepsilon^2)$.
	
	For the fixed $r>0$, we will use the decomposition $U=U_{-r}\cup(U\setminus U_{-r})$ to prove one term is ``zero value'' and the other term is ``zero measure''.  Using $0\le u_\varepsilon-u_0\le C\varepsilon$, we have
	
	\[
	\frac1\varepsilon\left|\int_U(u_\varepsilon-u_0)d\mu_{0,\varepsilon}^{(j)}\right|
	\le C\mu_{0,\varepsilon}^{(j)}(U_{-r}) +C\mu_{0,\varepsilon}^{(j)}(U\setminus U_{-r}).
	\]
	
	Note $r$ is fixed, Bedford--Taylor continuity \cite{BedfordTaylor1976Dirichlet,BedfordTaylor1982Capacity} on $U_{-r}\Subset U$ gives $\mu_{0,\varepsilon}^{(j)}\rightharpoonup V^{-1}\omega_{u_0}^n=0$ for the first term, so it tends to $0$.  For the second term, the uniform Hessian bounds of Theorem~\ref{intro:thm-uniform-regularity} make the measure bounded, and then estimating the region of integration via Proposition~\ref{prop:BH_to_collar} gives $\mu_{0,\varepsilon}^{(j)}(U\setminus U_{-r})\le Cr$.  After taking the limsup of the left-hand side and letting $r\downarrow0$, we prove the claim.
	
\end{proof}
\begin{proposition}[Localized first-order expansion]
As $\varepsilon\downarrow0$,
\begin{equation}\label{dyn:eq-localized-first-order-expansion}
E_{\omega_0}(u_\varepsilon)-E_{\omega_0}(u_0)=\frac{\varepsilon}{n+1}\sum_{j=0}^{n-1}\int_{\partial U}\left(-\partial_{\nu_{\mathrm{out}}}u_0\right)d\sigma_\varepsilon^{(j)}+o(\varepsilon),
\end{equation}
where $\sigma_\varepsilon^{(j)}:=\mu_{0,\varepsilon}^{(j)}\llcorner\partial U$.
\end{proposition}
\begin{proof}
We can insert the decomposition
\[
\mu_{0,\varepsilon}^{(j)} =\mu_{0,\varepsilon}^{(j)}\llcorner(X\setminus\partial U) +\sigma_\varepsilon^{(j)}
\]
into \eqref{dyn:eq-energy-algebraic}.  The first part is $o(\varepsilon)$ by Lemma~\ref{dyn:lem-bulk-shell-vanishing}.  Note that $\omega_{u_\varepsilon}^n=0$ in a neighborhood of $\partial U\subset U_\varepsilon$, so the $j=n$ term vanishes.  Lemma~\ref{dyn:lem-first-order-expansion} states that, uniformly on $\partial U$,
\[
u_\varepsilon-u_0 =\varepsilon\left(-\partial_{\nu_{\mathrm{out}}}u_0\right)+o(\varepsilon).
\]
Then, using the estimate $\|\sigma_\varepsilon^{(j)}\|\le \mu_{0,\varepsilon}^{(j)}(X)=V^{-1}\int_X\omega_{u_\varepsilon}^{j}\wedge\omega_{u_0}^{n-j}=1$, we obtain that the integral of the remainder is also $o(\varepsilon)$.  Finally, summing over $j$ gives the result.
\end{proof}

\subsection{Mixed flux currents and stability}\label{dyn:subsec-mixed-flux}
For $0\le j\le n-1$ define
\[
S_0^{(j)}:=\sum_{k=0}^{n-j-1}\omega_0^k\wedge\omega_{u_0}^{n-j-1-k}, \quad \lambda_\varepsilon^{(j)}:=\frac1V\omega_{u_\varepsilon}^j\wedge\omega_0^{n-j},
\]
and
\begin{equation}\label{dyn:eq-mixed-current}
\mathcal J_\varepsilon^{(j)}:=\frac1V\dc u_0\wedge\omega_{u_\varepsilon}^j\wedge S_0^{(j)}.
\end{equation}
\begin{proposition}\label{dyn:thm-finite-eps-mixed-flux}
Let $0<\varepsilon\le\varepsilon_1$.
For every $0\le j\le n-1$,
\[
d\mathcal J_\varepsilon^{(j)}=\mu_{0,\varepsilon}^{(j)}-\lambda_\varepsilon^{(j)}\quad\text{on }U, \qquad \sigma_\varepsilon^{(j)}=-\Tr_{\partial U}^{\mathrm{out}}\mathcal J_\varepsilon^{(j)}.
\]
\end{proposition}
\begin{proof}
Write $B_\varepsilon^{(j)}:=\omega_{u_\varepsilon}^j\wedge S_0^{(j)}$.
This is closed and has $L^\infty$ coefficients on $U$.
Apply the weak product rule and use the Bedford--Taylor factorization
\[
\ddc u_0\wedge S_0^{(j)}=\omega_{u_0}^{n-j}-\omega_0^{n-j}
\]
to obtain the divergence identity.

Set $\nu_\varepsilon^{(j)}:=\mu_{0,\varepsilon}^{(j)}-\lambda_\varepsilon^{(j)}$ and fix $\varphi\in C^\infty(X)$.
Because $u_0=0$ on $X\setminus U$, global Bedford--Taylor integration by parts gives
\begin{equation}\label{dyn:eq-ibtp-mixed}
\int_X\varphi\,d\nu_\varepsilon^{(j)}=\frac1V\int_Uu_0\,\ddc\varphi\wedge B_\varepsilon^{(j)}.
\end{equation}
The equality $\mu_{0,\varepsilon}^{(j)}=\lambda_\varepsilon^{(j)}$ on $X\setminus\overline U$ places the support of $\nu_\varepsilon^{(j)}$ in $\overline U$.
For each fixed $\varepsilon>0$, $\lambda_\varepsilon^{(j)}$ has $L^\infty$ density near $\Sigma\Subset U_\varepsilon$, and so $\nu_\varepsilon^{(j)}\llcorner\Sigma=\sigma_\varepsilon^{(j)}$.
We can now use Lemma~\ref{stat:lem-bt-gg-extraction} with $v=u_0|_U$ and $B=B_\varepsilon^{(j)}|_U$ to extract the boundary identity.
\end{proof}

\begin{theorem}[Weak-* stability of mixed traces]\label{thm:mixed_trace_stability}
For every $0\le j\le n-1$,
\[
\sigma_\varepsilon^{(j)}\rightharpoonup^*\sigma_{0+}^{(j)}:=-\Tr_{\partial U}^{\mathrm{out}}\mathcal J_0^{(j)}\qquad\text{in }\mathcal M(\partial U).
\]
Moreover,
\[
\sum_{j=0}^{n-1}\sigma_{0+}^{(j)}=\sigma_{\mathrm{tot}}=\left(-\partial_{\nu_{\mathrm{out}}}u_0\right)d\sigma_\rho^{\mathrm{tot}}.
\]
\end{theorem}
\begin{proof}
The convergence of the measures can be realized by viewing them as traces, and the convergence of the traces can be justified through the convergence of the functions on the inner levels.  Specifically, fix $U_{-r}\Subset U$; on it we have $u_\varepsilon\to u_0$ in $C^1$ by Proposition~\ref{reg:prop-quantitative-delta-domain-stability}.  Using Bedford--Taylor continuity \cite[Theorem~2.4 and Proposition~5.2]{BedfordTaylor1982Capacity}, we obtain the convergence of the mixed currents; then, using Lemma~\ref{stat:lem-strong-weak-flux} and Proposition~\ref{dyn:thm-finite-eps-mixed-flux}, we obtain
\[
\mathcal J_\varepsilon^{(j)}\rightharpoonup\mathcal J_0^{(j)},\qquad d\mathcal J_\varepsilon^{(j)}\rightharpoonup^*d\mathcal J_0^{(j)}\qquad\text{on }U_{-r}.
\]
Then, using Proposition~\ref{reg:lem-uniform-inner-C11} and Lemma~\ref{dyn:lem-trace-stability}, we obtain the desired result.
For the final measure identity, it suffices to note the combinatorial identity
\[
\sum_{j=0}^{n-1}\omega_{u_0}^j\wedge S_0^{(j)}=\sum_{j=0}^{n-1}\sum_{k=0}^{n-j-1}\omega_{u_0}^{n-1-k}\wedge\omega_0^k=\sum_{p=0}^{n-1}(p+1)\omega_{u_0}^p\wedge\omega_0^{n-1-p},
\]
together with Proposition~\ref{prop:total_jump_density}.  Hence Theorem~\ref{intro:thm-mixed-traces} is proved.
\end{proof}

\subsection{The Hadamard formula}\label{dyn:subsec-hadamard-formula}
\begin{theorem}\label{dyn:thm-closed-hadamard}
The right derivative exists and satisfies
\begin{equation}
\left.\frac{d}{d\varepsilon}E_{\omega_0}(u_\varepsilon)\right|_{\varepsilon=0^+} =\frac1{n+1}\int_{\partial U}\left(-\partial_{\nu_{\mathrm{out}}}u_0\right)d\sigma_{\mathrm{tot}},
\end{equation}
where $\sigma_{\mathrm{tot}}=\sum_{j=0}^{n-1}\sigma_{0+}^{(j)}$. Equivalently,
\begin{equation}\label{dyn:eq-closed-hadamard-square}
\left.\frac{d}{d\varepsilon}E_{\omega_0}(u_\varepsilon)\right|_{\varepsilon=0^+} =\frac1{n+1}\int_{\partial U}\left(-\partial_{\nu_{\mathrm{out}}}u_0\right)^2d\sigma_\rho^{\mathrm{tot}}.
\end{equation}
\end{theorem}
\begin{proof}
Divide \eqref{dyn:eq-localized-first-order-expansion} by $\varepsilon$ and apply Theorem~\ref{thm:mixed_trace_stability}.
Its summed trace identity gives both formulas.
\end{proof}

\begin{theorem}[Hadamard formula along the signed parallel family]
\label{dyn:thm-signed-family-hadamard}
Choose $\tau>0$ with $2\tau<\varepsilon_1$ (Convention~\ref{intro:conv-parameters}), and put $I_\tau=[-\tau,\tau]$.  For $t\in I_\tau$, set
\[
e(t):=E_{\omega_0}(u_t), \qquad F(t):=\frac1{n+1}\int_{\Sigma_t}q_t^2 \,d\sigma_{\rho,t}^{\mathrm{tot}}, \qquad q_t=-\partial_{\nu_t}u_t.
\]
Then $F\in C^{0,1/2}(I_\tau)$ and
\begin{equation}\label{dyn:eq-signed-energy-fundamental}
e(t)-e(s)=\int_s^tF(r)\,dr, \qquad s,t\in I_\tau.
\end{equation}
In particular, $e\in C^{1,1/2}(I_\tau)$ and, for $-\tau<t<\tau$,
\begin{equation}
\frac{d}{dt}E_{\omega_0}(u_t) =\frac1{n+1}\int_{\Sigma_t}q_t\,d\sigma_{\mathrm{tot},t} =\frac1{n+1}\int_{\Sigma_t}q_t^2 \,d\sigma_{\rho,t}^{\mathrm{tot}}.
\end{equation}
At the endpoints the precise one-sided identities are
\[
D_+e(-\tau)=F(-\tau),\qquad D_-e(\tau)=F(\tau).
\]
\end{theorem}

\begin{proof}
Since the normal flow gives the identification, using \eqref{reg:eq-normal-derivative-half-holder} and Corollary~\ref{stat:cor-parallel-density-bounds} we obtain
\[
\left\|\Psi_{s,t}^*(q_t^2d\sigma_{\rho,t}^{\mathrm{tot}}) -q_s^2d\sigma_{\rho,s}^{\mathrm{tot}}\right\|_{C^0(\Sigma_s,g_0)} \le C|t-s|^{1/2}.
\]
Then, using the uniform bound on the areas of the integration domains $\Sigma_t$, we also have
\begin{equation}\label{dyn:eq-F-half-holder}
	|F(t)-F(s)|\le C|t-s|^{1/2}.
\end{equation}
Let $s<t$.  A direct computation gives
\[
0 \le e(t) - e(s) \le \frac{1}{n+1} \sum_j \|u_t - u_s\|_{L^\infty(X)} \cdot 1 = \|u_t - u_s\|_{L^\infty(X)} \le C|t - s|,
\]
where the last step uses Lemma~\ref{reg:lem-linf-approx}.  Hence $e$ is Lipschitz continuous and absolutely continuous on $I_\tau$.  Fix $t_0\in[-\tau,\tau)$.  Applying Theorem~\ref{dyn:thm-closed-hadamard} at $t_0$ gives $D_+e(t_0)=F(t_0)$.  Therefore, $e'=F$ almost everywhere.  Using absolute continuity, we obtain \eqref{dyn:eq-signed-energy-fundamental}.  Since $F\in C^{0,1/2}(I_\tau)$, we get $e\in C^{1,1/2}(I_\tau)$.  The assertions at the endpoints follow by taking one-sided limits in the integral representation.  This proves Theorem~\ref{intro:thm-hadamard}.
\end{proof}

\paragraph{Application to the Darvas distance.}\label{dyn:par-darvas}
The boundary integral in Theorem~\ref{dyn:thm-signed-family-hadamard} also measures distance in $\mathcal E^1(X,\omega_0)$, the space of finite-energy $\omega_0$-psh potentials.  We use the Darvas distance $d_1$, normalized by the tangent norm $\|\xi\|_{1,v}=\int_X|\xi|\,\MA(v)$ on smooth K\"ahler potentials.  For the connected components $X_a$ of $X$, this means $d_1=\sum_a(V_a/V)d_{1,a}$, where $V_a=\int_{X_a}\omega_0^n$ and $d_{1,a}$ uses the volume $V_a$.

\begin{corollary}[Darvas distance along parallel deformations]\label{dyn:cor-darvas}
With $e$ and $F$ as in Theorem~\ref{dyn:thm-signed-family-hadamard}, for $-\tau\le s<t\le\tau$,
\[
d_1(u_s,u_t)=e(t)-e(s)=\int_s^tF(r)\,dr.
\]
The curve has metric speed
\[
|\dot u|_{d_1}(t):=\lim_{h\to0}\frac{d_1(u_{t+h},u_t)}{|h|}=F(t),\qquad -\tau<t<\tau,
\]
which is $1/2$-H\"older continuous, with the corresponding one-sided limits at the endpoints.
\end{corollary}
\begin{proof}
Since $\Omega_t\subset\Omega_s$, each competitor for $u_s$ is admissible for $u_t$, so $u_s\le u_t$.  Both potentials belong to $\mathcal E^1$ by boundedness.  Their rooftop envelope is therefore $u_s$, and \cite[Corollary~4.14]{Darvas2015FiniteEnergy} gives $d_1(u_s,u_t)=E_{\omega_0}(u_t)-E_{\omega_0}(u_s)$.  Applying this identity on each component and summing with weights $V_a/V$ gives the same formula on $X$.  Equation~\eqref{dyn:eq-signed-energy-fundamental} proves the integral identity.  For either sign of $h$, the distance quotient equals $|e(t+h)-e(t)|/|h|$.  Since $e'=F\ge0$ and $F\in C^{0,1/2}(I_\tau)$, its limit is $F(t)$; the endpoint statements follow in the same way.
\end{proof}

\subsection{Examples}

\begin{dynexample}{Trace, flux and Hadamard check on
\texorpdfstring{$\mathbb P^1$}{P1}}\label{dyn:ex-cp1-verification} On the affine chart $z=re^{i\theta}$, let
\[
\omega_0=\ddc\log(1+|z|^2) =\frac{2r}{(1+r^2)^2}\,dr\wedge d\theta, \qquad V=2\pi.
\]
For $U_R=\{|z|<R\}$,
\[
u_R=\max\!\left\{\log\frac{1+R^2}{1+|z|^2},0\right\}.
\]
On $U_R$, $\omega_{u_R}=0$ and $\dc u_R=-r^2(1+r^2)^{-1}d\theta$.  Since $n=1$, $\mathcal J[u_R]=\mathcal J[u_R;\mathrm{tot}]=V^{-1}\dc u_R$.  The normalized Monge--Amp\`ere measure has total mass one.  Its exterior part has mass $V^{-1}\int_{\{r>R\}}\omega_0=(1+R^2)^{-1}$, while its interior part vanishes.  Direct computation of the boundary mass and the interior flux gives
\[
\mu_R^{\mathrm{sing}}(\partial U_R) =-\Tr_{\partial U_R}^{\mathrm{out}}\mathcal J[u_R](\partial U_R) =\frac{R^2}{1+R^2}.
\]
Furthermore $-\partial_{\nu_{\mathrm{out}}}u_R=\sqrt2R$ (the factor $\sqrt2$ is the metric normalization $|\partial_r|_{g_0}=\sqrt2/(1+R^2)$ at $\partial U_R$), so the trace side is
\[
\frac12\int_{\partial U_R}(-\partial_{\nu_{\mathrm{out}}}u_R)\,d\sigma_{\mathrm{tot}} =\frac{R^3}{\sqrt2(1+R^2)}.
\]
A direct energy computation gives
\[
E_{\omega_0}(u_R)=\frac12\left(\log(1+R^2)-\frac{R^2}{1+R^2}\right), \qquad d\varepsilon=\frac{\sqrt2}{1+R^2}\,dR,
\]
and differentiation with respect to normal distance yields the same value, including the factor $1/2$.
For $R_1<R_2$, the ordered-potential identity used in Corollary~\ref{dyn:cor-darvas} gives
$d_1(u_{R_1},u_{R_2})=\frac12\left[\log\frac{1+R_2^2}{1+R_1^2}-\frac{R_2^2}{1+R_2^2}+\frac{R_1^2}{1+R_1^2}\right]$.

\end{dynexample}

\begin{dynexample}{Toric blow-up model: boundary mass and mixed Hadamard flux}\label{dyn:ex-toric-collapse}
Let $X=\operatorname{Bl}_p\mathbb P^n$ ($n\ge2$), with exceptional divisor $E$ and the strict transform $H_\infty$ of the hyperplane at infinity.  On $X\setminus(E\cup H_\infty)\simeq\mathbb C^n\setminus\{0\}$ put $t=\log|z|^2$ and normalize $\int_{S^1}\dc t=2\pi$ and $\int_{\mathbb P^{n-1}}\omega_{FS}^{n-1}=(2\pi)^{n-1}$; set $C_n:=(2\pi)^n$.  Choose a Calabi-type K\"ahler form \cite{Calabi1982Extremal,HwangSinger2002Momentum}
\[
\omega_0=m'(t)dt\wedge\dc t+m(t)\omega_{FS},\qquad m'>0,\quad m(-\infty)=b\in(0,1),\quad m(+\infty)=1.
\]
Smoothness across $E$ gives $m(t)-b=O(e^t)$, and orbit integration gives
\[
V=C_n\int_{-\infty}^{\infty}n m'(t)m(t)^{n-1}dt =C_n(1-b^n), \qquad g_0(\partial_t,\partial_t)=\frac{m'(t)}2.
\]

Fix $t_R$, let $U_R$ be $\{t<t_R\}$ together with $E$, and put $a:=m(t_R)$.  The comparison principle identifies the envelope as
\[
u_R(t)=
\begin{cases}
\displaystyle\int_t^{t_R}(m(s)-b)\,ds,&t<t_R,\\
0,&t\ge t_R,
\end{cases}
\]
The displayed formula extends continuously to $E$.  Its radial moment is $b$ on $U_R$ and $m(t)$ outside.  It is therefore $\omega_0$-psh; on $U_R$, $\omega_{u_R}=b\omega_{FS}$ and $\omega_{u_R}^n=0$.  At $\partial U_R$,
\[
\nu_{\mathrm{out}}=\sqrt{\frac2{m'(t_R)}}\,\partial_t, \qquad -\partial_{\nu_{\mathrm{out}}}u_R =\frac{a-b}{\sqrt{m'(t_R)/2}}.
\]
Here $\sqrt{2/m'(t_R)}$ is the unit-normal conversion coming from $g_0(\partial_t,\partial_t)=m'/2$; the same normalization produces the $\sqrt2$ factor in Example~\ref{dyn:ex-cp1-verification}.

For $\omega_q=q'dt\wedge\dc t+q\omega_{FS}$, orbit integration gives $\int_{\{t_1<t<t_2\}}\omega_q^n=C_n(q(t_2)^n-q(t_1)^n)$.  Apply this identity to smooth approximations of the radial moment of $u_R$ and pass to the limit.  The jump from $b$ to $a$ then yields
\[
\mu_R^{\mathrm{sing}}(\partial U_R) =\frac{C_n}{V}(a^n-b^n).
\]
The interior flux gives the same static measure:
\[
-\Tr^{\mathrm{out}}_{\partial U_R}\mathcal J[u_R](\partial U_R) =\frac{C_n(a-b)}V\sum_{p=0}^{n-1}b^pa^{n-1-p} =\frac{C_n}{V}(a^n-b^n).
\]
In contrast, for $\mathcal J[u_R;\mathrm{tot}]:=V^{-1}\dc u_R\wedge\sum_{p=0}^{n-1}(p+1)\omega_{u_R}^p\wedge\omega_0^{n-1-p}$ and $\sigma_{\mathrm{tot},R}:=-\Tr^{\mathrm{out}}_{\partial U_R}\mathcal J[u_R;\mathrm{tot}]$,
\[
\sigma_{\mathrm{tot},R}(\partial U_R) =\frac{C_n(a-b)}{V} \sum_{p=0}^{n-1}(p+1)b^pa^{n-1-p}.
\]

Let $U_{R,\varepsilon}$ be the outward normal deformation and write its boundary as $t=t_R(\varepsilon)$.  Then
\begin{equation}\label{dyn:eq-toric-normal-parameter}
t_R(0)=t_R, \qquad t_R'(0) =\frac1{\sqrt{m'(t_R)/2}}.
\end{equation}
Since $(a-b)\sum_{p=0}^{n-1}(p+1)b^pa^{n-1-p}=\sum_{k=0}^{n-1}a^{n-k}b^k-nb^n$, the trace prediction is
\begin{equation}\label{dyn:eq-toric-trace-prediction}
\frac1{n+1}\int_{\partial U_R} (-\partial_{\nu_{\mathrm{out}}}u_R)\,d\sigma_{\mathrm{tot},R} =\frac{C_n(a-b)}{(n+1)V\sqrt{m'(t_R)/2}} \left(\sum_{k=0}^{n-1}a^{n-k}b^k-nb^n\right).
\end{equation}

In a separate energy calculation, the boundary terms vanish.  Expanding the mixed powers and integrating by parts once gives
\begin{align*}
E_{\omega_0}(u_R) &=\frac{C_n}{(n+1)V}\int_{-\infty}^{t_R} u_Rm'\sum_{q=1}^n q m^{q-1}b^{n-q}\,dt\\
&=\frac{C_n}{(n+1)V}\int_{-\infty}^{t_R} (m-b)\left(\sum_{q=1}^n m^qb^{n-q}-nb^n\right)dt.
\end{align*}
Therefore
\[
\frac{d}{dt_R}E_{\omega_0}(u_R) =\frac{C_n(a-b)}{(n+1)V} \left(\sum_{k=0}^{n-1}a^{n-k}b^k-nb^n\right),
\]
and \eqref{dyn:eq-toric-normal-parameter} yields exactly \eqref{dyn:eq-toric-trace-prediction}.

The static mass contains $a^n-b^n$, while the variational trace contains $\sum_{k=0}^{n-1}a^{n-k}b^k-nb^n$.  For $n=2$, these become $a^2-b^2$ and $a^2+ab-2b^2$, so the two fluxes already differ in the first higher-dimensional case.  At $\varepsilon=0$, the literal measures $\mu_{0,0}^{(j)}:=V^{-1}\omega_{u_R}^j\wedge\omega_{u_R}^{n-j}$ all equal $\mu_R$, and their boundary restrictions sum to $n\mu_R^{\mathrm{sing}}$.  The one-sided limiting traces sum to $\sigma_{\mathrm{tot},R}$. This is the boundary-restriction discontinuity described above.
\end{dynexample}

\section*{Acknowledgments}
The author thanks Akito Futaki for helpful discussions and suggestions, and Professor Xiaohua Zhu for pointing out the connection between the results of this paper and the Darvas distance $d_1$.

\appendix
\numberwithin{equation}{section}
\section{Proofs of uniform regularity estimates}\label{app:sec-proofs-regularity}

\subsection{External estimates and their use}\label{app:subsec-imported}
Set $D=U_{\varepsilon,m}$ and $u=u_{\varepsilon,\delta,m}$. The geometric data are those of Lemma~\ref{reg:lem-smooth-parallel-geometry}, and the first-order bound is $\|u\|_{C^{0,1}}\le G$. All estimates are used on the smooth approximants, with constants uniform in $(\varepsilon,\delta,m)$ in the indicated density range.
\begin{enumerate}[label=\textbf{(E\arabic*)},ref=E\arabic*,leftmargin=*,itemsep=.15\baselineskip,topsep=0pt,parsep=0pt]
\item\label{app:input-blocki}\textbf{Tangential--normal boundary estimate (B{\l}ocki).}
The tangential--normal part of \cite[Theorem~3.2 and the proof of Theorem~3.2']{Blocki2012Geodesics} gives
\[
 |\nabla^2u(T,\nu)|\le C_{\rm B}\qquad\text{on }\partial D
\]
for unit tangent $T$ and outward unit normal $\nu$. Here the subsolution is zero, its positive reference form is $\omega_0$, the boundary value is zero, and $\|\delta^{1/n}\|_{C^{0,1}}\le1$. B{\l}ocki explicitly records the density-independent tangential and mixed components before treating the double-normal derivative. Their constants are controlled by the common physical-boundary charts, the reference form, and $G$.

\item\label{app:input-chu}\textbf{Core real-Hessian propagation (Chu).}
The local auxiliary-function calculation in \cite[\S\S4A, 4C and Proposition~6.1]{Chu2021C11}, applied to $u-\sup_{D'}u$ on $D'=D$ or $D_r=\{\varrho\ge r\}$, gives
\[
 1+\sup_{D'}|\nabla^2u|\le C_{\rm Ch}\bigl(1+\sup_{\partial D'}|\nabla^2u|\bigr).
\]
In Chu's notation $f=\delta$, so $df^{1/n}=\nabla^2f^{1/n}=0$. At an interior maximum the calculation uses the ambient geometry and the common oscillation and gradient bounds. At a boundary maximum it retains the boundary Hessian. The oscillation comparison in Step~4b below gives the displayed linear dependence; its auxiliary function is defined on the whole core, so $C_{\rm Ch}$ is uniform in $r$.

\item\label{app:input-ds}\textbf{Affine fields and boundary comparisons (Dinew--Sroka).}
We use the directional quotient of \cite[\S4]{DinewSroka2025Krylov} and the affine fields and comparisons of \cite[Proposition~5.7, Lemmas~5.8--5.10, and Theorem~5.11]{DinewSroka2025Krylov}. In Subsection~\ref{app:subsec-proof-boundary}, Claim~1 verifies the differential inequality and physical-face expansion in uniform holomorphic charts, Claim~2 proves the cap-to-normal estimate, and Claim~3 supplies the cap bound. The local potentials and density roots enter through their uniform $C^{3,1}$ and $C^{1,1}$ bounds, respectively.
\end{enumerate}
\subsection{Proof of Lemma~\ref{reg:lem-uniform-collar-core-bridge}}\label{app:subsec-proof-bridge}

\begin{proof}[Proof of Lemma~\ref{reg:lem-uniform-collar-core-bridge}]
The boundary reduction, normalized collar quotient, and core estimate give \eqref{reg:eq-weak-interior-bridge}. The auxiliary parameters are chosen in Step~5, uniformly in $(\varepsilon,\delta,m)$.

\noindent\textbf{Step 0. Boundary reduction to $M_\delta$.}
Constants in this proof depend only on $G$ and the uniform geometry in Lemma~\ref{reg:lem-smooth-parallel-geometry}.  Differentiating $u|_{\partial D}=0$ twice in tangential directions gives
\[
 \nabla^2u(T,T')=-(\partial_\nu u)\operatorname{II}(T,T')  \quad\text{on }\partial D.
\]
The uniform $C^1$ bound and the uniform control of $\operatorname{II}$ in Lemma~\ref{reg:lem-smooth-parallel-geometry} therefore give $|\nabla^2u(T,T')|\le C$ on $\partial D$.  Input~\ref{app:input-blocki} separately gives $|\nabla^2u(T,\nu)|\le C$.  Together,
\begin{equation}
 |\nabla^2u(T,T')|+|\nabla^2u(T,\nu)|\le C \quad\text{on }\partial D.
 \label{reg:eq-old-boundary}
\end{equation}
Thus the only uncontrolled boundary component is the double-normal one measured by $M_\delta$ in \eqref{reg:eq-Mdelta}.

\noindent\textbf{Step 1. Normalized collar linearization.}
We set $\widetilde\omega=\omega_0+\ddc u$ and introduce three new quantities:
\begin{equation}
 S_0=\tr_{\widetilde\omega}\omega_0,  \qquad a^{i\bar j}=\frac{\widetilde g^{i\bar j}}{S_0},  \qquad Lf=a^{i\bar j}\nabla_i\nabla_{\bar j}f.
 \label{reg:eq-normalized-linearization}
\end{equation}
Here $a$ is positive definite with $\tr_{g_0}a=1$; feeding these two properties into the arithmetic--geometric mean inequality gives
\begin{equation}
 Lu=\frac n{S_0}-1\le\delta^{1/n}-1\le-\frac12.
 \label{reg:eq-Lu-negative}
\end{equation}
On the common boundary collar we set $r:=\rho_m-\varepsilon<0$ and $\varrho:=c_\varrho(1-e^{\Lambda r})$, so that $\varrho>0$ in $D$, $\varrho=0$ on $\partial D$, and $\ddc\varrho=-c_\varrho\Lambda e^{\Lambda r}(\ddc r+\Lambda\,dr\wedge\dc r)$. For fixed $\Lambda\gg1$, the tangential Levi term and the rank-one normal term give $\ddc\varrho\le-c\omega_0$ on a common collar.  Contracting with $a\ge0$, using $\tr_{g_0}a=1$, and scaling $c_\varrho$, we find $R_0>0$ such that
\begin{equation}
 L\varrho\le-2\quad\text{on }\mathcal C_{R_0}:=\{0<\varrho<R_0\}
 \label{reg:eq-Lrho-negative}
\end{equation}
for every $a>0$ with $\tr_{g_0}a=1$.  All these choices are uniform in $(\varepsilon,m)$.

\noindent\textbf{Step 2. Lift to the direction bundle.}
At a point, diagonalize $a=\operatorname{diag}(\lambda_k)$ in the unitary frame $Z_k=(e_k-iJe_k)/\sqrt2$, write $(E_1,\ldots,E_{2n})=(e_1,Je_1,\ldots,e_n,Je_n)$, and extend $a$ to the real bilinear form $A$ with
\[
 A(e_k,e_k)=A(Je_k,Je_k)=\frac{\lambda_k}{2},  \qquad \tr_{\mathbb R}A=1,  \qquad A^{ab}\nabla_a\nabla_b=L.
\]
We use the curvature convention
\[
 \operatorname{Rm}(X,Y)Z  :=\nabla_X\nabla_YZ-\nabla_Y\nabla_XZ-\nabla_{[X,Y]}Z.
\]
Because $(X,\omega_0)$ is K\"ahler, the Chern and Levi--Civita connections agree.  For a function $F$ on $T^{1,0}D$, a curve $\gamma$ with $\gamma'(0)=X$, and $\zeta(s)$ parallel along $\gamma$, we define the two derivative operators
\[
 \mathcal H_XF(x,\zeta)  :=\left.\frac d{ds}\right|_{s=0}F(\gamma(s),\zeta(s)),  \qquad  D_Y^{\rm v}F(x,\zeta)  :=\left.\frac d{ds}\right|_{s=0}F(x,\zeta+sY^{1,0})
\]
for real $Y$.  Their complex components $\mathcal H_i,\mathcal H_{\bar j}$ and $\mathcal D^{\rm v}_i,\mathcal D^{\rm v}_{\bar j}$ separate the two chain-rule directions of $F(x,\zeta(x))$: $\mathcal H$ moves the base point while keeping $\zeta$ parallel, whereas $D^{\rm v}$ fixes the base point and modifies the fibre vector.  For a complex fibre vector $\zeta$, the real vector $V=\zeta+\bar\zeta$ satisfies $|V|^2=2|\zeta|^2$.  Writing $\zeta=\zeta^kZ_k$, we use the lowered conjugate components $\bar\zeta_j=g_{j\bar k}\overline{\zeta^k}$ and $\zeta_{\bar j}=g_{k\bar j}\zeta^k$, and we set
\[
 t=\varrho+\beta,  \quad p=\partial\varrho(\zeta),  \quad \mathfrak r=\frac{\alpha\bar p}{t},  \quad q_j=\frac1t\left(\frac{\bar\zeta_j}{4}+\frac{\alpha\varrho_j\bar p}{t}\right),  \quad B_A(V)=A^{ab}\operatorname{Rm}(E_a,V)E_b.
\]
with $(aq)^i=a^{i\bar j}q_{\bar j}$ and $q_{\bar j}=\overline{q_j}$.
\begin{samepage}
Following \cite[(4.6)--(4.11)]{DinewSroka2025Krylov}, we define the reduced-fibre covariant operator
\begin{align}
 \mathcal PF={}&\operatorname{Re}\!\left\{a^{i\bar j}\bigl(  \mathcal H_i\mathcal H_{\bar j}F +\mathfrak r\mathcal H_i\mathcal D^{\rm v}_{\bar j}F +\bar{\mathfrak r}\mathcal D^{\rm v}_i\mathcal H_{\bar j}F +|\mathfrak r|^2\mathcal D^{\rm v}_i\mathcal D^{\rm v}_{\bar j}F\bigr)\right\}  \notag\\
 &-2\operatorname{Re}\bigl((aq)^i\mathcal D^{\rm v}_iF\bigr)  +D^{\rm v}_{B_A(V)}F.
 \label{reg:eq-covariant-P}
\end{align}
\end{samepage}
Freezing coefficients, the four second-order terms expand $(\mathcal H_i+\bar{\mathfrak r}\mathcal D^{\rm v}_i)(\mathcal H_{\bar j}+\mathfrak r\mathcal D^{\rm v}_{\bar j})$, with principal block
\[
 \begin{pmatrix}a&\mathfrak r a\\ \bar{\mathfrak r}a&|\mathfrak r|^2a\end{pmatrix}
 =\begin{pmatrix}1\\ \bar{\mathfrak r}\end{pmatrix}
  \begin{pmatrix}1&\mathfrak r\end{pmatrix}\!\otimes a\ge0,
\]
hence the pointwise maximum-principle inequality applies for each $\delta>0$. The coercivity estimates below use positivity and trace normalization.

\noindent\textbf{Step 3. Coercivity of the numerator and denominator.}
Adapting the Hessian quotient in \cite[(4.23), (4.31)--(4.40)]{DinewSroka2025Krylov}, we set $T=\nabla^2u$ and $G=\nabla u$, and define
\[
 \mathcal S_A(T)=A^{ab}\sum_cT(E_a,E_c)T(E_b,E_c),  \quad W=T(V,V)+K|\zeta|^2|G|^2,  \quad \mathcal V=t^{-\alpha}|p|^2+\alpha^{-1}t^{1-\alpha}|\zeta|^2.
\]
The correction term in $W$ makes $\mathcal PW$ supply the coercive $K|\zeta|^2\mathcal S_A(T)$ term, using only $\tr_{\mathbb R}A=1$.  The collar weight $\mathcal V$ balances $p$ and $\zeta$ at scale $t$, and $L\varrho\le-2$ gives $-\mathcal P\mathcal V>0$; thus $W/\mathcal V$ is the boundary-adapted quotient for the maximum principle, using positivity and trace normalization.  There are fixed $K\ge1$, $\alpha_0>0$, $0<R_1\le R_0$, and $c>0$ such that, for $0<\alpha\le\alpha_0$, $0<R\le R_1$, and $0<\beta\le R$, one has on $T\mathcal C_R$
\begin{align}
 \mathcal PW&\ge cK|\zeta|^2\mathcal S_A(T)-C_{\alpha,\beta,K}|\zeta|^2,
 \label{reg:eq-numerator-coercive}\\
 -\mathcal P\mathcal V&\ge ct^{-\alpha}  \left(\frac{|\zeta|^2}{\alpha}+\frac{\alpha}{t}|p|^2\right).
 \label{reg:eq-denominator-coercive}
\end{align}
The constants depend on the indicated auxiliary parameters, the fixed geometry, and the uniform first-order bound.

We first verify the numerator.  For the curvature convention above, Hessian Codazzi reads
\[
 (\nabla_XT)(Y,Z)-(\nabla_YT)(X,Z)  =-\langle\operatorname{Rm}(X,Y)Z,G\rangle.
\]
Extend $X,Y,V$ so that their covariant derivatives vanish at the point in question.  Applying Codazzi twice and commuting the two derivatives acting on $T$ gives
\begin{align}
 (\nabla_X\nabla_YT)(V,V)={}&(\nabla_V\nabla_VT)(X,Y)  -2T(\operatorname{Rm}(X,V)Y,V) \notag\\
 &-T(Y,\operatorname{Rm}(X,V)V)   -T(X,\operatorname{Rm}(Y,V)V) \notag\\
 &-\big\langle(\nabla_X\operatorname{Rm})(Y,V)V  +(\nabla_V\operatorname{Rm})(X,V)Y,G\big\rangle.
 \label{reg:eq-fourth-covariant-commutator}
\end{align}
Since $\log(\det\widetilde g/\det g_0)=\log\delta$, twice differentiating along a geodesic extension of $V$ yields
\[
 A^{ab}(\nabla_V\nabla_VT)(E_a,E_b)=\mathcal Q_V,  \qquad  \mathcal Q_V  =\frac1{S_0}\left|  \widetilde g^{-1/2}(\nabla_V\widetilde g)\widetilde g^{-1/2}  \right|_{\rm HS}^{2}\ge0.
\]
Contracting \eqref{reg:eq-fourth-covariant-commutator} with $A$ gives
\begin{align}
 A^{ab}\mathcal H_a\mathcal H_b[T(V,V)]  ={}&\mathcal Q_V-2T(B_A(V),V)  -2A^{ab}T(E_b,\operatorname{Rm}(E_a,V)V)  -\langle C_A(V),G\rangle,
 \label{reg:eq-numerator-commutator}
\end{align}
where $\mathcal Q_V\ge0$ and
\[
 C_A(V)=A^{ab}\bigl((\nabla_{E_a}\operatorname{Rm})(E_b,V)V  +(\nabla_V\operatorname{Rm})(E_a,V)E_b\bigr),  \qquad |C_A(V)|\le C|V|^2.
\]
Moreover, $D^{\rm v}_{B_A(V)}[T(V,V)]=2T(B_A(V),V)$, so the complete vertical drift in \eqref{reg:eq-covariant-P} cancels the second term exactly. Differentiating the equation once and then applying Hessian Codazzi gives
\[
 a^{i\bar j}(\nabla_{Z_i}T)(V,\bar Z_j)  =-a^{i\bar j}\langle\operatorname{Rm}(Z_i,V)\bar Z_j,G\rangle,  \qquad  \left|a^{i\bar j}(\nabla_{Z_i}T)(V,\bar Z_j)\right|  \le C|G|\,|V|.
\]
Since $\tr_{g_0}a=1$, the $g_0$-eigenvalues of $a$ are nonnegative and sum to $1$, so $0<a\le g_0^{-1}$.  Thus every remaining curvature or $q$-drift term linear in $T$ is bounded by $C_{\alpha,\beta}|\zeta|^2\mathcal S_A(T)^{1/2}$; the vertical--vertical contraction contains only $Lu$, which lies in $(-1,-1/2]$.  Since $|\mathfrak r|\le C\alpha\beta^{-1}|\zeta|$, $|q|\le C(\beta^{-1}+\alpha\beta^{-2})|\zeta|$, and $L|G|^2=2\mathcal S_A(T)+O(1)$, weighted Young inequalities in the positive quadratic form $A$ give, for every $0<\eta<1$,
\begin{align}
 \mathcal P[T(V,V)]  &\ge-\eta|\zeta|^2\mathcal S_A(T)       -C_{\eta,\alpha,\beta}|\zeta|^2,
 \label{reg:eq-numerator-part}\\
 \mathcal P(|\zeta|^2|G|^2)  &\ge(2-\eta)|\zeta|^2\mathcal S_A(T)       -C_{\eta,\alpha,\beta}|\zeta|^2.
 \label{reg:eq-gradient-part}
\end{align}
Hence
\[
 \mathcal PW\ge\bigl((2-\eta)K-\eta\bigr)|\zeta|^2\mathcal S_A(T)  -C_{\eta,\alpha,\beta,K}|\zeta|^2.
\]
Fixing, for example, $\eta=1/4$ and then any fixed $K\ge1$ proves \eqref{reg:eq-numerator-coercive}.  The inverse metric in $\mathcal Q_V$ occurs in a nonnegative square; discarding that square leaves constants uniform in $\delta$.

For the denominator we set $h_i=\varrho_i$, $|h|_A^2=a^{i\bar j}h_i h_{\bar j}$, $|\zeta|_A^2=a^{i\bar j}\zeta_j\bar\zeta_i$, and $C_{\zeta,h,p}=\operatorname{Re}(a^{i\bar j}\zeta_jh_i\bar p)$.  Pairing the commutator in $p$ and $\bar p$ gives
\begin{align}
 &\operatorname{Re}\!\left\{a^{i\bar j}\bigl[  (\mathcal H_i\mathcal H_{\bar j}p)\bar p  +p(\mathcal H_i\mathcal H_{\bar j}\bar p)\bigr]\right\}  \notag\\
 &\qquad=\operatorname{Re}\!\left\{a^{i\bar j}\bigl[  (\nabla_\zeta\nabla^2\varrho)_{i\bar j}\bar p  +(\nabla_{\bar\zeta}\nabla^2\varrho)_{i\bar j}p\bigr]\right\}  -2\operatorname{Re}(\bar p\,\partial\varrho(\mathcal B_A\zeta)).
 \label{reg:eq-paired-commutator}
\end{align}
where $\mathcal B_A\zeta=(B_A(V))^{1,0}$.  In the diagonal frame,
\[
 \mathcal B_A\zeta  =\frac12\sum_k\lambda_k\operatorname{Rm}(\bar Z_k,\zeta)Z_k.
\]
The K\"ahler type identities give
\[
 a^{i\bar j}\mathcal H_i\mathcal H_{\bar j}p  =a^{i\bar j}(\nabla_\zeta\nabla^2\varrho)_{i\bar j},  \qquad  a^{i\bar j}\mathcal H_i\mathcal H_{\bar j}\bar p  =a^{i\bar j}(\nabla_{\bar\zeta}\nabla^2\varrho)_{i\bar j}  -2\overline{\partial\varrho(\mathcal B_A\zeta)},
\]
which proves \eqref{reg:eq-paired-commutator}.  Moreover,
\[
 D^{\rm v}_{B_A(V)}(t^{-\alpha}|p|^2)  =2t^{-\alpha}\operatorname{Re}    \bigl(\bar p\,\partial\varrho(\mathcal B_A\zeta)\bigr),
\]
so the curvature rows cancel exactly.  Substituting $\mathfrak r,q$, while retaining for the moment the vertical drift acting on the angular summand $\alpha^{-1}t^{1-\alpha}|\zeta|^2$, we obtain
\begin{align}
 \mathcal P\mathcal V  -D^{\rm v}_{B_A(V)}   \bigl(\alpha^{-1}t^{1-\alpha}|\zeta|^2\bigr)  ={}&E_\varrho +\frac{1-\alpha}{\alpha}t^{-\alpha}L\varrho|\zeta|^2 +\alpha t^{-\alpha-1}(1+L\varrho)|p|^2 \notag\\
&-(1-\alpha)t^{-\alpha-1}|h|_A^2|\zeta|^2 -\alpha t^{-\alpha-2}|p|^2|h|_A^2 \notag\\
&-\left(\frac12+2\alpha\right)t^{-\alpha-1}C_{\zeta,h,p} -\frac1{2\alpha}t^{-\alpha}|\zeta|_A^2,
 \label{reg:eq-denominator-expansion}
\end{align}
where
\begin{align*}
 E_\varrho={}&t^{-\alpha}a^{i\bar j}\bigl(  \varrho_{\zeta i\bar j}\bar p+\varrho_{\bar\zeta i\bar j}p  +\varrho_{\zeta i}\varrho_{\bar\zeta\bar j}  +\varrho_{\zeta\bar j}\varrho_{\bar\zeta i}\bigr)\\
 &-\alpha t^{-\alpha-1}a^{i\bar j}\bigl(  h_{\bar j}\varrho_{\bar\zeta i}p+h_i\varrho_{\zeta\bar j}\bar p\bigr),
\end{align*}
with $\varrho_{\zeta i}:=(\nabla^2\varrho)(\zeta,Z_i)$, $\varrho_{\zeta i\bar j}:=(\nabla_\zeta\nabla^2\varrho)_{i\bar j}$, and the analogous barred notation.  Put
\[
 s=|\zeta|^2,\qquad x=|h|_A,\qquad y=|\zeta|_A,  \qquad P_0=|p|.
\]
Since $|p|\le C|\zeta|$, $\tr_{g_0}a=1$, and the collar has a uniform $C^{3,1}$ norm, there is a fixed $C_0$ such that
\begin{equation}
 |E_\varrho|  \le C_0t^{-\alpha}\left(s+\frac{\alpha}{t}P_0x\sqrt{s}\right),  \qquad |C_{\zeta,h,p}|\le P_0xy.
 \label{reg:eq-denominator-errors}
\end{equation}
By \eqref{reg:eq-Lrho-negative},
\[
 \frac{1-\alpha}{\alpha}L\varrho\,s  +\frac{\alpha}{t}(1+L\varrho)P_0^2  \le-2\frac{1-\alpha}{\alpha}s-\frac{\alpha}{t}P_0^2.
\]
The two required weighted Young inequalities are
\[
 C_0\frac{\alpha}{t}P_0x\sqrt{s}  \le\frac{1-\alpha}{2t}x^2s       +\frac{C_1\alpha^2}{t}P_0^2,
\]
where $C_1=C_1(C_0)$ is fixed, and, once $\alpha_0$ is small enough that $\frac12+2\alpha_0\le2^{-1/2}$,
\[
 \left(\frac12+2\alpha\right)\frac{P_0xy}{t}  \le\frac{\alpha}{2t^2}P_0^2x^2+\frac1{4\alpha}y^2.
\]
Shrink $\alpha_0$ further so that $C_0\le(1-\alpha)/\alpha$ and $C_1\alpha^2\le3\alpha/4$ for $0<\alpha\le\alpha_0$.  Substitution in \eqref{reg:eq-denominator-expansion} gives
\[
 \mathcal P\mathcal V  -D^{\rm v}_{B_A(V)}   \bigl(\alpha^{-1}t^{1-\alpha}|\zeta|^2\bigr)  \le  -\frac{1-\alpha}{\alpha}t^{-\alpha}|\zeta|^2  -\frac\alpha4t^{-\alpha-1}|p|^2  -\frac\alpha2t^{-\alpha-2}|p|^2|h|_A^2  -\frac1{4\alpha}t^{-\alpha}|\zeta|_A^2.
\]
The only uncanceled part of the complete drift is its action on the angular term, and
\[
 \left|D^{\rm v}_{B_A(V)}  \bigl(\alpha^{-1}t^{1-\alpha}|\zeta|^2\bigr)\right|  \le C\alpha^{-1}t^{1-\alpha}|\zeta|^2.
\]
Since $t\le2R_1$, shrink $R_1\le R_0$ so that $2CR_1\le(1-\alpha_0)/2$.  The first negative row absorbs the remaining drift and proves \eqref{reg:eq-denominator-coercive}, the covariant reduced-fibre analogue of \cite[Lemma~4.3, (4.22)]{DinewSroka2025Krylov}.

\noindent\textbf{Step 4a. Quotient estimate on the collar.}
Let $r_b=\theta R$, $\Gamma_s=\{\varrho=s\}$, and set
\[
 H_{r_b}=1+\sup_{\{\varrho\ge r_b\}}|\nabla^2u|,  \qquad H_{\Gamma_{r_b}}=1+\sup_{\Gamma_{r_b}}|\nabla^2u|,  \qquad P_*=\sup_{\mathcal C_R}|\partial\varrho|^2.
\]
Applying the maximum principle to the degree-zero quotient $W/\mathcal V$ on the unit sphere bundle gives
\begin{equation}
 H_{\Gamma_{r_b}}  \le Cq_{\rm col}H_{r_b}  +C\beta^\alpha K_{r_b}M_\delta+C_{R,\alpha,\beta},
 \label{reg:eq-collar-interface}
\end{equation}
where
\begin{align*}
 K_{r_b}&=\alpha^{-1}(r_b+\beta)^{1-\alpha}  +P_*(r_b+\beta)^{-\alpha},\\
 q_{\rm col}&=x_0^{1-\alpha}  +\frac{\alpha P_*}{R+\beta}x_0^{-\alpha},  \qquad x_0=\frac{r_b+\beta}{R+\beta}.
\end{align*}
At an interior positive maximum, \eqref{reg:eq-numerator-coercive} and \eqref{reg:eq-denominator-coercive} bound the quotient by a constant independent of $H_{r_b}$ and $M_\delta$.  On the outer level $\Gamma_R$ the quotient is at most $C\alpha(R+\beta)^{\alpha-1}H_{r_b}+C$, whereas on the physical boundary \eqref{reg:eq-old-boundary} gives $W/\mathcal V\le C\beta^\alpha M_\delta+C_{\alpha,\beta}$.  Multiplying by $\mathcal V\le K_{r_b}|\zeta|^2$ on $\Gamma_{r_b}$ gives the upper estimate in \eqref{reg:eq-collar-interface}.  Applying the same argument to $JV$, together with admissibility, gives the lower estimate.

\noindent\textbf{Step 4b. Core estimate and feedback removal.}
The feedback term in \eqref{reg:eq-collar-interface} is removed by applying Input~\ref{app:input-chu} on the core
\[
 D_{r_b}:=\{\varrho\ge r_b\},\qquad \partial D_{r_b}=\Gamma_{r_b}.
\]
If $H_{r_b}\le H_0$, where $H_0$ is a sufficiently large constant depending only on the fixed geometry and the uniform first-order bound, then \eqref{reg:eq-Chu-core} holds after enlarging $C_{\rm core}$. We may therefore assume $H_{r_b}>H_0$. Choose $H_0$ large enough that
\[
 |\widetilde\omega|_{g_0}^2\le 2n+4|\nabla^2u|_{g_0}^2\le5H_{r_b}^2
 \qquad\text{throughout }D_{r_b}.
\]
Following \cite[\S4A, (4-1)--(4-3)]{Chu2021C11}, normalize $u$ by subtracting its supremum, let $\lambda_1$ be the largest eigenvalue of $\nabla^2u$, and define $Q=\log\lambda_1+Q_{\rm rem}$ on $\{\lambda_1>0\}\cap D_{r_b}$.  The term $h_1(|\widetilde\omega|_{g_0}^2)$ with
\[
 h_1(s)=-\frac13\log\bigl(10H_{r_b}^2-s\bigr)
\]
satisfies, since $|\widetilde\omega|_{g_0}^2\le5H_{r_b}^2$,
\[
 \operatorname{osc}_{D_{r_b}}  h_1\bigl(|\widetilde\omega|_{g_0}^2\bigr)  \le\frac13\log2,
\]
and the other terms of $Q_{\rm rem}$ have uniformly bounded oscillation. The eigenvalue comparison recorded in Input~\ref{app:input-chu} gives $x_*$ with $\lambda_1(x_*)\ge cH_{r_b}$; comparison with a maximum point $y$ of $Q$ yields
\[
 H_{r_b}\le C\bigl(1+\lambda_1(y)\bigr).
\]
If $y$ lies in the interior, then $\lambda_1(y)\le C$, while for $y\in\Gamma_{r_b}$ it gives $\lambda_1(y)\le H_{\Gamma_{r_b}}$. Consequently
\begin{equation}
 H_{r_b}\le C_{\rm core}(1+H_{\Gamma_{r_b}}).
 \label{reg:eq-Chu-core}
\end{equation}
Here $C_{\rm core}$ is uniform for $0<r_b\le R$ and independent of $(\delta,\varepsilon,m)$, by the uniformity recorded in Input~\ref{app:input-chu}.

\noindent\textbf{Step 5. Parameter order and small-coefficient extraction.}
Choose $R,\theta,r_b,\alpha,\beta$ in that order.  Fix $R\le R_1$ and choose $\theta>0$ so that $6CC_{\rm core}\theta<1/2$ and $r_b:=\theta R<r_*$.  Next take
\[
 \alpha\le\min\left\{\alpha_0,\frac{R\theta}{P_*},  \frac{\log2}{\log(2/\theta)}\right\}.
\]
For $\beta\le r_b$ this yields $q_{\rm col}\le6\theta$.  Combining \eqref{reg:eq-collar-interface} and \eqref{reg:eq-Chu-core} and absorbing the feedback yields
\[
 H_{r_b}\le C\beta^\alpha K_{r_b}M_\delta+C_{R,\alpha,\beta}.
\]
Since $\beta^\alpha K_{r_b}\to0$ as $\beta\downarrow0$, once $R,\theta,\alpha$ are fixed, for every $\eta_0>0$ one can choose $\beta$, independently of $\delta$, so that
\begin{equation}
 \sup_{\{\varrho\ge r_b\}}|\nabla^2u|  \le\eta_0M_\delta+C_{\eta_0,r_b}.
 \label{reg:eq-weak-interior-bridge}
\end{equation}
\end{proof}

\subsection{Proof of Proposition~\ref{reg:prop-uniform-degenerate-boundary}}\label{app:subsec-proof-boundary}

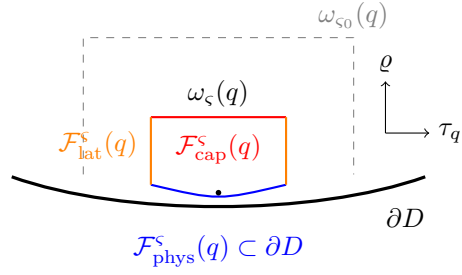
\begin{figure}[!htbp]
\centering
\begin{tikzpicture}[scale=1.05]
\draw[very thick] (-2.6,0) .. controls (-1.2,-0.5) and (1.2,-0.5) .. (2.6,0);
\node[below] at (2.35,-0.28) {$\partial D$};
\draw[gray,dashed] (-1.7,0.03) -- (-1.7,1.75) -- (1.7,1.75) -- (1.7,0.03);
\node[gray,above] at (1.7,1.75) {$\omega_{\varsigma_0}(q)$};
\draw[blue,thick] (-0.85,-0.1) .. controls (0,-0.3) .. (0.85,-0.1);
\draw[red,thick] (-0.85,0.75) -- (0.85,0.75);
\draw[orange,thick] (-0.85,-0.1) -- (-0.85,0.75);
\draw[orange,thick] (0.85,-0.1) -- (0.85,0.75);
\node at (0,1.05) {$\omega_\varsigma(q)$};
\fill (0,-0.2) circle (0.035);
\node[blue,below] at (0,-0.62) {$\mathcal F^{\varsigma}_{\mathrm{phys}}(q)\subset\partial D$};
\node[red,below] at (0,0.66) {$\mathcal F^{\varsigma}_{\mathrm{cap}}(q)$};
\node[orange,left] at (-0.9,0.4) {$\mathcal F^{\varsigma}_{\mathrm{lat}}(q)$};
\draw[->] (2.1,0.55) -- (2.1,1.2);
\node[above] at (2.1,1.2) {$\varrho$};
\draw[->] (2.1,0.55) -- (2.65,0.55);
\node[right] at (2.65,0.55) {$\tau_q$};
\end{tikzpicture}
\caption{The adapted three-face cylinder $\omega_\varsigma(q)\subset\omega_{\varsigma_0}(q)$ in tangential--normal coordinates $(\tau_q,\varrho)$.}
\label{fig:three-face-cylinder}
\end{figure}

\begin{proof}[Proof of Proposition~\ref{reg:prop-uniform-degenerate-boundary}]
For $n=1$, the equation is $\Delta_{g_0}u=2(\delta-1)$ with zero boundary value. The Dirichlet Schauder estimate \cite[Chapter~6]{GilbargTrudinger2001Elliptic} gives $\|u\|_{C^{2,\alpha}}\le C(\|u\|_{C^0}+2|\delta-1|)$, uniformly over the common boundary charts and fixed background coefficients. For $n\ge2$ and $\delta\ge2^{-n}$, \cite[Theorem~3.2]{Blocki2012Geodesics} gives the boundary Hessian estimate with this fixed density floor, and Input~\ref{app:input-chu} gives the whole-domain bound. We henceforth take $n\ge2$ and $0<\delta\le2^{-n}$. Write $M_\delta$ as in \eqref{reg:eq-Mdelta}, and use the collar function $\varrho$ and radius $R_0$ from Lemma~\ref{reg:lem-uniform-collar-core-bridge}.

At a boundary point $q$, bounded holomorphic linear and quadratic coordinate changes put the inward side in the form $x_n>F_q(\tau)$, where $\tau=(\operatorname{Re}z',\operatorname{Im}z',y_n)$ and
\[
 F_q(\tau)=\tfrac12|\tau|^2+R_q(\tau),\qquad |D^jR_q(\tau)|\le C|\tau|^{3-j}\quad(0\le j\le2),\qquad \|F_q\|_{C^{3,1}}\le C.
\]
The linear change normalizes the tangential Levi form; a holomorphic quadratic change of $z_n$ removes the pure and mixed quadratic terms and sets the $y_n^2$ coefficient to $1/2$. The common gradient and Levi lower bounds control these changes and their inverses on a common radius. Choose nested chart neighborhoods $V_j\Subset W_j$ covering the physical boundaries and perform all recentering near $V_j$ inside $W_j$, with one fixed local potential. Put $r_q=x_n-F_q(\tau)$, $N_q=\partial_{x_n}$, and $U=u+\Phi$. Then
\[
 \det(U_{i\bar j})=\delta\det(g_{0,i\bar j}),\qquad U|_{\partial D}=\Phi.
\]
The potentials have uniform $C^{3,1}$ bounds, and the density roots have uniform $C^{1,1}$ bounds. At each center, writing the inward $g_0$-unit normal as $\gamma$, the tangential and mixed estimates \eqref{reg:eq-old-boundary} give $U_{N_qN_q}=\alpha_q^2u_{\gamma\gamma}+O(1)$, with $0<c\le\alpha_q\le C$. Thus, if $M_E$ is one plus the supremum of the positive coordinate double-normal derivatives at chart centers,
\begin{equation}
 C_A^{-1}M_\delta\le M_E\le C_AM_\delta.
 \label{reg:eq-normal-dictionary}
\end{equation}
Admissibility and tangency of $J\nu$ give $u_{\nu\nu}\ge-C$ on $\partial D$. Input~\ref{app:input-chu} and \eqref{reg:eq-old-boundary} therefore yield
\begin{equation}
 1+\sup_D|\nabla^2_{g_0}u|\le C_{\rm wh}\bigl(1+\sup_{\partial D}|\nabla^2_{g_0}u|\bigr)\le C_{\rm wh}M_\delta\le C_{\rm wh}C_AM_E.
 \label{reg:eq-whole-domain-coarse}
\end{equation}
Adding the local potential and changing coordinates gives $1+\sup_{D\cap W_j}|D^2U_j|\le CM_E$. This bounds the artificial faces at the fixed outer scale.

\noindent\textbf{Claim 1. Affine-field calculation and physical-face expansion.}
For each real unit tangent vector $e$ at $q$, the fields in \cite[Proposition~5.7]{DinewSroka2025Krylov} give a real field $T=e+A_ez$ with holomorphic affine flow $\phi_t$, with $A_e$ skew-symmetric, commuting with the complex structure, and $A_eN_q=-e$. The spherical quadratic jet above gives $Tr_q=O(|\tau|^2)$ on $\partial D$. These properties and all coefficients are uniform in $q,e,\varepsilon,m$.

Set $H=(U_{i\bar j})$, $F(H)=(\det H)^{1/n}=f$, and $\mathcal L v=H^{i\bar j}v_{i\bar j}/\tr H^{-1}$. Since $\tr F'(H)=F(H)\tr H^{-1}/n\ge1$, differentiating along the unitary affine flow of $T$, using concavity of $F$, and differentiating in the constant direction $N_q$ give
\begin{equation}
 \begin{gathered}
 \mathcal L(T^2U)\ge-C_d,\qquad |\mathcal L(N_qU)|\le C_d,\\
 \mathcal L\varrho\le-\kappa,\quad 0\le\mathcal L Q\le1,\quad 0\le\mathcal L(Q^2)\le4Q,
 \end{gathered}
 \label{app:eq-affine-differential}
\end{equation}
where $Q=|\tau|^2$. Indeed, twice differentiating $F(D^2_{\mathbb C}(U\circ\phi_t))=f\circ\phi_t$ gives $F'(H)[D^2_{\mathbb C}(T^2U)]\ge T^2f$, and $\|f\|_{C^{1,1}}\le C$. The remaining inequalities follow from the strict complex Hessian of $-\varrho$ and the unit trace of $H^{-1}/\tr H^{-1}$.

On the physical boundary write $du=\lambda\,dr_q$, where $\lambda=N_qu$ has a uniform Lipschitz bound by \eqref{reg:eq-old-boundary}. Decompose $T=Y+b\nu_E$, with $\nu_E=Dr_q/|Dr_q|$, $Y$ tangent, and $b=Tr_q/|Dr_q|=O(|\tau|^2)$. Differentiating $u=0$ along $Y$ gives
\[
 \begin{aligned}
 T^2u-\lambda T^2r_q
 & =2b\bigl(D^2u(Y,\nu_E)-\lambda D^2r_q(Y,\nu_E)\bigr)\\
 &\quad +b^2\bigl(D^2u(\nu_E,\nu_E)-\lambda D^2r_q(\nu_E,\nu_E)\bigr).
 \end{aligned}
\]
The two brackets are bounded by $C$ and $CM_E$, respectively. Since $T^2r_q(q)=0$, the linear Taylor term in $\tau$ of $T^2\Phi+\lambda(q)T^2r_q$ supplies a real affine function $\ell_q$, with $\ell_q(q)=N_q\ell_q=0$, such that
\begin{equation}
 \widehat w_q:=T^2U-T^2U(q)-\ell_q,\qquad
 |\widehat w_q|\le C|\tau|^2+CM_E|\tau|^4\quad\text{on }\partial D.
 \label{reg:eq-local-physical-quartic}
\end{equation}
Here the constant Taylor term is $T^2U(q)$, and the Lipschitz variation of $\lambda$ times $T^2r_q=O(|\tau|)$ contributes $O(|\tau|^2)$. Also $\mathcal L\widehat w_q\ge-C_d$ because $\ell_q$ is pluriharmonic. This is the physical-face expansion of \cite[Lemma~5.8, (5.46)]{DinewSroka2025Krylov} in the present charts.

Fix $c_*>0$ and $0<\varsigma_0<1$ so that the three-face cylinders
\[
 \omega_s(q)=\{0<\varrho<c_*s^4,\ |\tau_q|<s\}\subset D\cap W_j,\qquad 0<s\le\varsigma_0,
\]
lie in the common collar and $c_*\varsigma_0^4<R_0$. Their faces are $\mathcal F_{\rm phys}^s=\{\varrho=0\}$, $\mathcal F_{\rm cap}^s=\{\varrho=c_*s^4\}$, and $\mathcal F_{\rm lat}^s=\{|\tau_q|=s\}$, with the other inequalities retained. On the fixed outer cylinder compare $\widehat w_q$ with $P=aQ+bM_EQ^2+dM_E\varrho$. Choose $a,b$ to dominate \eqref{reg:eq-local-physical-quartic} and the coarse bound on the lateral face, then $d$ to dominate the cap and to satisfy $d\kappa\ge C_d+a+4b\varsigma_0^2+1$. Equation~\eqref{app:eq-affine-differential} gives $\mathcal L(\widehat w_q-P)\ge0$, so
\begin{equation}
 \widehat w_q\le C_4|\tau_q|^2+C_4M_E|\tau_q|^4+C_4M_E\varrho\quad\text{on }\omega_{\varsigma_0}(q).
 \label{reg:eq-fixed-scale-lateral}
\end{equation}
This reproduces the comparison in \cite[Lemma~5.9, (5.47)]{DinewSroka2025Krylov}, with $C_4$ fixed by the outer geometry.

\noindent\textbf{Claim 2. Conditional cap-to-normal closure.}
Suppose that, for some $0<\varsigma<\varsigma_0/8$, every adapted chart satisfies
\begin{equation}
 \sup_{\mathcal F_{\rm cap}^{\varsigma}(q)}|D^2U|\le c_0\varsigma^5M_E+C_{{\rm cap},\varsigma},\qquad C_{{\rm cap},\varsigma}\ge1,
 \label{reg:eq-local-cap-input}
\end{equation}
with a common additive constant. Set
\begin{equation}
 C_\varsigma=C_{\rm loc}+A_0\varsigma^{-4}C_{{\rm cap},\varsigma},\qquad A_\varsigma=\varsigma M_E+C_\varsigma.
 \label{reg:eq-inflated-cap-constant}
\end{equation}
Choose $\eta>0$ so that $\mathcal L(\varrho+\eta Q)\le-\kappa/2$, and then $K$ so that $K\kappa/2\ge C_d+1$. We compare $\widehat w_q$ with $KA_\varsigma(\varrho+\eta Q)$. On the physical and lateral faces, respectively, \eqref{reg:eq-local-physical-quartic} and \eqref{reg:eq-fixed-scale-lateral} give
\[
 \widehat w_q\le CQ+CM_E\varsigma^2Q,\qquad
 \widehat w_q\le C\varsigma^2+CM_E\varsigma^4.
\]
The term $K\eta A_\varsigma Q$ dominates these bounds when $C_{\rm loc}$ is large and $\varsigma$ is below a fixed threshold. On the cap, the gradient and affine-correction bounds give
\[
 \widehat w_q\le C_t c_0\varsigma^5M_E+C_tC_{{\rm cap},\varsigma}
 \le KA_\varsigma c_*\varsigma^4
\]
after choosing $c_0\le Kc_*/(2C_t)$ and $A_0\ge2C_t/(Kc_*)$. Equation~\eqref{app:eq-affine-differential} gives $\mathcal L(\widehat w_q-KA_\varsigma(\varrho+\eta Q))\ge0$. The maximum principle and inward differentiation at $q$ yield
\begin{equation}
 N_q(T^2U)(q)\le C_N A_\varsigma.
 \label{reg:eq-local-DS-normal-derivative}
\end{equation}
This is the small-cap comparison of \cite[Lemma~5.10, (5.50)--(5.56)]{DinewSroka2025Krylov}.

To pass to a fixed normal direction, put $h=\partial_\gamma u$ on $\partial D$. Its $C^1$ norm is controlled by the mixed estimate. At every center $p$, the identity $A_eN_p=-e$ gives
\[
 T^2(N_pU)(p)=N_p(T^2U)(p)+2D^2U(e,e)(p)+dU(A_ee)(p)\le CA_\varsigma.
\]
Quadratic tangency makes this the second derivative of the boundary restriction along the tangential projection of $T$. On that boundary, $N_pU=\alpha_ph+\beta_p$, where $\alpha_p=|dr_p|_{g_0}^{-1}$ is uniformly positive and $\beta_p=N_p\Phi_p$; both have uniform $C^2$ bounds. Expanding their tangential derivatives gives $\nabla^2_{\partial D}h\le CA_\varsigma g_{\partial D}$. In the fixed chart at $q$, the same identity and Taylor's formula give
\[
 N_qU(z)-N_qU(q)-a_q\cdot\tau_q(z)\le C_FA_\varsigma|\tau_q(z)|^2,\qquad |a_q|\le C,
\]
on the physical face. This supplies the change of normal used in \cite[(5.62)--(5.72)]{DinewSroka2025Krylov}.

Keep the chart and $N_q$ fixed, and on $\omega_{\varsigma_0}(q)$ set
\[
 B_q=N_qU-N_qU(q)-a_q\cdot\tau_q-K_fA_\varsigma(\varrho+\eta Q).
\]
Choose $K_f\eta\ge C_F$ and $K_f\kappa/2\ge C_d+1$. Then $\mathcal L B_q\ge0$ and $B_q\le0$ on the physical face. On the two artificial faces, $\varrho+\eta Q\ge d_0:=\min\{c_*\varsigma_0^4,\eta\varsigma_0^2\}>0$. The gradient bound controls the first three terms of $B_q$ by a fixed $C_g$; increasing $C_{\rm loc}$ so that $K_fC_{\rm loc}d_0\ge C_g$ gives $B_q\le0$ there as well. Hence $B_q\le0$ in the cylinder and $U_{N_qN_q}(q)\le CA_\varsigma$. If $M_\delta>1$, take $q$ at its double-normal maximum and use the center comparison preceding \eqref{reg:eq-normal-dictionary}; for $M_\delta=1$, \eqref{reg:eq-normal-dictionary} directly gives the same bound. We obtain
\begin{equation}
 M_E\le C_*\bigl(\varsigma M_E+C_\varsigma\bigr).
 \label{reg:eq-local-quartic-output}
\end{equation}
This is the fixed-normal comparison of \cite[(5.73)--(5.75)]{DinewSroka2025Krylov}. The coefficient $C_*$ depends on the fixed geometry, $G$, $K_f$, and the normal comparisons. The cap constant enters through the additive term $C_\varsigma$.

\noindent\textbf{Claim 3. Cap verification and absorption.}
Fix the chart data and the constants in Claims~1--2. Choose $\varsigma$ below their threshold so that $c_*\varsigma^4<R_0$ and $C_*\varsigma<1/2$. Apply Lemma~\ref{reg:lem-uniform-collar-core-bridge} with
\[
 r_*:=\frac{c_*\varsigma^4}{2},\qquad \eta_0:=\frac{c_0\varsigma^5}{2C_AC_{\rm pot}},
\]
where $C_{\rm pot}$ bounds the coordinate--covariant Hessian comparison. It gives $0<r_b<r_*$ and \eqref{reg:eq-weak-interior-bridge}. Every cap is a level of the same function $\varrho$, with $\varrho=c_*\varsigma^4>r_b$. Thus
\[
 \sup_{\mathcal F_{\rm cap}^{\varsigma}(q)}|D^2U|
 \le C_{\rm pot}(\eta_0M_\delta+C_{\eta_0,r_b})+C_\Phi+C_{\rm ov}
 \le \tfrac12c_0\varsigma^5M_E+C_{{\rm cap},\varsigma},
\]
where $C_{{\rm cap},\varsigma}:=C_{\rm pot}C_{\eta_0,r_b}+C_\Phi+C_{\rm ov}+1$ is common to the charts and uniform in $(\varepsilon,\delta,m)$. Here $C_\Phi$ bounds the potentials and $C_{\rm ov}$ their bounded changes on overlaps. This verifies \eqref{reg:eq-local-cap-input}. Absorption in \eqref{reg:eq-local-quartic-output} gives $M_E\le2C_*C_\varsigma$, and then \eqref{reg:eq-normal-dictionary}, \eqref{reg:eq-old-boundary}, and \eqref{reg:eq-whole-domain-coarse} give the boundary and whole-domain Hessian bounds.

Finally, positivity of the boundary block gives $0\le c^*B^{-1}c<d$. The scalar $d$ is a background term plus a fixed multiple of $u_{\nu\nu}+u_{J\nu,J\nu}$, so $d\le C$. All geometric and auxiliary constants above are fixed uniformly over $(\varepsilon,\delta,m)$.
\end{proof}

\begingroup
\footnotesize
\apptocmd{\thebibliography}{\setlength{\itemsep}{0pt}}{}{}
\bibliographystyle{alpha}
\bibliography{reference_for_submission}
\endgroup
\end{document}